\documentclass[a4paper,english]{article}

\usepackage{authblk}
\usepackage{amsfonts}
\usepackage{amsmath}
\usepackage{amsthm}
\usepackage{graphicx}
\usepackage{lipsum}  
\usepackage{multicol, blindtext}
\usepackage{subfig}
\usepackage{tabularx}
\usepackage{float}
\usepackage{breqn}

\def \be{\begin{equation}}
\def \ee{\end{equation}}
\def\ie{\textit{i.e.}}
\newlength\figheight
\newlength\figwidth

\usepackage{epsfig} 
\usepackage{amssymb} 
\usepackage{amsmath} 
\usepackage{widetext}
\usepackage{todonotes}
\usepackage{graphics,graphicx}
\usepackage{epsfig}

\usepackage{amscd}
\usepackage{subfig}
\usepackage{todonotes}

\definecolor{darkred}{rgb}{.7,0,0}

\definecolor{darkgreen}{rgb}{0,0.7,0}

\definecolor{darkblue}{rgb}{0,0,0.7}

\newcommand{\bp}{\begin{proof}}
\newcommand{\ep}{\end{proof}}

\newtheorem{theo}{Theorem}[section]
\newtheorem{theorem}[theo]{Theorem}
\newtheorem{lemma}[theo]{Lemma}

\newtheorem{corollary}[theo]{Corollary}
\newtheorem{remark}[theo]{Remark}
\newtheorem{definition}[theo]{Definition}
\newtheorem{example}[theo]{Example}

\newtheorem{assumption}{Assumption}
\newcommand{\bes}{\begin{equation*}}
\newcommand{\ees}{\end{equation*}}
\newcommand{\bea}{\begin{eqnarray}}
\newcommand{\eea}{\end{eqnarray}}

\newcommand{\R}{\mathbb{R}}

\newcommand{\norm}[1]{| #1 |}

\newcommand{\cF}{\mathcal{F}}
\newcommand{\E}{\mathbb{E}}

\def \IR{\mathbb R}

\def \IE{\mathbb E}

\newcommand{\xf}[2]{x^#1_#2 }
\newcommand{\yc}[2]{x^#1_#2  }
\usepackage[lined,boxed,commentsnumbered]{algorithm2e}

\newcommand{\LRARR}[4]{{\mbox{\raise 0.4 mm \hbox{$#1$}}} \;
  \mathop{\stackrel{\displaystyle\longrightarrow}\longleftarrow}^{#3}_{#4}
  \; {\mbox{\raise 0.4 mm\hbox{$#2$}}}}

\RequirePackage[OT1]{fontenc}

\global\long\def\nobs{N}

\global\long\def\iid{\overset{\text{i.i.d.}}{\sim}}

\global\long\def\sobs{s}

\newcommand{\RR}{\mathbb{R}}
\usepackage[square,numbers]{natbib}

\begin{document}
\title{Multilevel Monte Carlo methods for the approximation of invariant measures of stochastic differential equations}

\author[3]{Michael B. Giles}
\author[1]{Mateusz B. Majka }
\author[2]{Lukasz Szpruch}
\author[1]{Sebastian J. Vollmer}
\author[2]{Konstantinos C. Zygalakis }
\affil[1]{Mathematical Institute, University of Oxford}
\affil[2]{Department of Statistics, University of Warwick}
\affil[3]{School of Mathematics, University of Edinburgh}





  \maketitle

\begin{abstract}
We develop a framework that allows the use of the multi-level
Monte Carlo (MLMC) methodology \cite{Giles2015Acta} to calculate expectations
with respect to the invariant measure of an ergodic SDE.  In that context, we
study the (over-damped) Langevin equations with a strongly concave potential. We 
show that, when appropriate contracting couplings for the numerical 
integrators are available,   one can obtain a uniform in time estimate of 
the MLMC variance in contrast to the majority of the results in the MLMC 
literature.  As a consequence,  a root mean square error of 
$\mathcal{O}(\varepsilon)$ is achieved with  $\mathcal{O}(\varepsilon^{-2})$ complexity on par 
with Markov Chain Monte Carlo  (MCMC)  methods, which however can be 
computationally intensive when applied to large data sets.  
Finally, we  present a multi-level version of the recently introduced 
Stochastic Gradient Langevin Dynamics (SGLD) method  \cite{welling2011bayesian}  built 
for large datasets applications. We show that this is the first stochastic 
gradient  MCMC method with complexity   $\mathcal{O}(\varepsilon^{-2}|\log
{\varepsilon}|^{3})$,  in contrast to the complexity 
$\mathcal{O}(\varepsilon^{-3})$ of currently available methods. Numerical experiments confirm our 
theoretical findings.

%
\end{abstract}

\section{Introduction}

We consider a probability measure $\pi(dx)$ with a density $\pi(x)\propto e^{U(x)}$ on $\RR^d$ with an unknown normalising constant. A typical task is the computation of the following quantity 
   \begin{equation} \label{eq:expi}
   \pi(g):=\E_\pi g = \int_{\RR^d} g(x) \pi(dx),  \quad g\in L^1(\pi).
   \end{equation}
Even if $\pi(dx)$ is given in an explicit form, quadrature methods, in general, are inefficient in high dimensions. 
On the other hand probabilistic methods scale very well with the dimension and are often the method of choice. With this in mind, we explore the connection between dynamics of stochastic differential equations (SDEs) 
\begin{equation} \label{eq:langevin}
dX_{t}=  \nabla U(X_{t})dt+\sqrt{2}dW_{t}, \quad X_{0} \in \IR^{d},
\end{equation}
and the target probability measure $\pi(dx)$. The key idea is that under appropriate assumptions on $U(x)$ one can show that the solution to \eqref{eq:langevin} is ergodic and has $\pi(dx)$  as its unique invariant measure \cite{Has80}.  
However, there exists only a limited number of cases where analytical solutions to 
\eqref{eq:langevin} are available and typically some form of approximation needs to be  employed. 

 The numerical analysis  approach \cite{KlP92} is to 
discretize \eqref{eq:langevin} and run the corresponding Markov chain for a long time interval. One drawback of the numerical analysis  approach is that it might be the case that even though 
\eqref{eq:langevin} is geometrically ergodic, the corresponding numerical discretization might not be  \cite{RT96}, while in addition  extra care is required when $\nabla U$ is not globally Lipschitz \cite{MSH02,DT02,RT96,TSS00,HJP11}.  The numerical analysis approach also introduces  bias 
because the numerical invariant measure does not coincide  with the exact one  in general \cite{TaT90,AVZ14}, resulting hence in a 
biased estimation of $\pi(g)$ in \eqref{eq:expi}.  Furthermore, if one uses the Euler-Maruyama method to discretize  \eqref{eq:langevin}, then 
computational complexity\footnote{In this paper the computational complexity is measured in terms of the expected number of random number 
generations and arithmetic operations.} of  $\mathcal{O}(\varepsilon^{-3})$ is required for  achieving a root mean square error of order $\mathcal{O}(\varepsilon)$ 
in the  approximation of $\eqref{eq:expi}$. Furthermore, even if one mitigates the bias due to numerical discretization by a series of decreasing time steps in combination with an  appropriate weighted time average  of the quantity of interest \cite{DP02}, the computational complexity still remains $\mathcal{O}(\varepsilon^{-3})$  \cite{Teh2016sgld}. 

%


An alternative way of sampling from $\pi(dx)$ exactly, so that it does not face the bias issue introduced 
by pure discretisation of \eqref{eq:langevin}, is by using the Metropolis-Hastings algorithm \cite{Ha70}.  We will refer to this as the computational statistics approach. The fact that the Metropolis Hastings algorithm leads 
to asymptotically unbiased samples of the probability measure is one 
of the reasons  why it has been the method of choice in 
computational statistics. Moreover,  unlike the numerical analysis approach, computational complexity   of  $\mathcal{O}(\varepsilon^{-2})$ is now required for  achieving root mean square error  of order $
\mathcal{O}(\varepsilon)$ in the (asymptotically unbiased) approximation of $\eqref{eq:expi}$.
We notice that MLMC \cite{Giles2015Acta} and the unbiasing scheme
\cite{Rhee2012,Rhee2015,glynn2014exact} are able to achieve the $\mathcal{O}(\varepsilon^{-2})$ complexity for
computing expectations of SDEs on a fixed time interval $[0,T]$, despite using biased numerical discretisations. We are interested in extending this approach to the case of ergodic SDEs on the time interval $[0, \infty)$, see also discussion in  \cite{Giles2015Acta}.
 

A particular application of \eqref{eq:langevin} is when
one is interested in approximating the posterior expectations for a
Bayesian inference problem. More precisely, if for a
fixed parameter $x$ the data $\left\{ y_{i}\right\} _{i=1,\dots,N}$
are i.i.d. with densities ${\pi(y_i|x)}$, then $\nabla U(x)$ becomes 
\begin{equation}
\nabla U(x)=\nabla\log{\pi_{0}(x)}+\sum_{i=1}^{N}\nabla\log{\pi(y_{i}|x)},\label{eq:bay_iid}
\end{equation}
with $\pi_{0}(x)$ being the prior distribution of $x$. 
When dealing with problems where the number of data items $N \gg 1$ is large, both the standard numerical analysis and the MCMC approaches 
suffer due to the high computational cost associated with calculating the likelihood terms $\nabla \log{\pi(y_{i}|x)}$ over each data item $y_{i}$. 
One way to circumvent this problem
is the Stochastic Gradient Langevin Dynamics algorithm (SGLD)
introduced in \cite{welling2011bayesian}, which replaces the sum of the  $N$ likelihood terms by an appropriately reweighted sum of $s \ll N$ terms. 
This leads to the following recursion formula 
\begin{equation} \label{eq:SGLD}
x_{k+1} = x_{k}+h\left(\nabla\log{\pi_{0}(x_{k})}+\frac{N}{s}\sum_{i=1}^{s}\nabla\log{\pi(y_{\tau_{i}^{k}}\vert x_{k})}\right) 
+ \sqrt{2h}\xi_{k}
\end{equation}
where $\xi_{k}$ is a standard Gaussian random variable on $\mathbb{R}^{d}$
and $\tau^{k}=(\tau_{1}^{k},\cdots,\tau_{s}^{k})$ is a random subset
of $[N]=\{1,\cdots,N\}$, generated for example by sampling with or
without replacement from $[N]$. Notice, that this corresponds to a noisy Euler discretisation, which  for  fixed $N,s$ still has  computational complexity $\mathcal{O}(\varepsilon^{-3})$  as discussed in \cite{Teh2016sgld,teh2014sgldB}. In this article, we are able to show that careful coupling between fine and coarse paths allows the application of the MLMC framework and hence reduction of the computational complexity of the algorithm to $\mathcal{O}(\varepsilon^{-2}|\log{\varepsilon}|^{3})$. We also remark that coupling in time has been recently further developed in \cite{fang2016adaptive,fang2017adaptive,FANG2019} for Euler schemes. 



We would like to stress that in our analysis of the computational complexity of MLMC for SGLD we treat $N$ and $s$ as fixed parameters. Hence our results show that in cases in which one is forced to consider $s \ll N$ samples (e.g.\ in the big data regime, where the cost of taking into account all $N$ samples is prohibitively large) MLMC for SGLD can indeed reduce the computational complexity in comparison to the standard MCMC. However, recently the authors of \cite{truecost} have argued that for the standard MCMC the gain in complexity of SGLD due to the decreased number of samples can be outweighed by the increase in the variance caused by subsampling. We believe that an analogous analysis for MLMC would be highly non-trivial and we leave it for future work.

In summary the main contributions of this paper are:
\begin{enumerate}
\item Extension of the MLMC framework to the time interval $[0,\infty)$ for \eqref{eq:langevin} when $U$ is  strongly concave.  
\item A convergence theorem that allows the estimation of the MLMC  variance using uniform in time estimates in the $2$- Wasserstein metric for a variety of different numerical methods. 
\item A new method of estimation of expectations with respect to the invariant measures without the need of accept/reject steps (as in MCMC). The methods we propose can be better parallelised than MCMC, since computations on all levels can be performed independently.
\item   The application of this scheme to stochastic gradient Langevin dynamics (SGLD) which reduces the complexity of $\mathcal{O}(\varepsilon^{-3})$ to $ \mathcal{O}(\varepsilon^{-2}\left|\log\varepsilon\right|^{3})$ much closer to the standard $\mathcal{O}(\varepsilon^{-2})$  complexity of MCMC. 
\end{enumerate}

The rest of the paper is organised as follows. In Section \ref{sec:prel} we describe the standard MLMC framework, discuss the contracting properties of the true trajectories of \eqref{eq:langevin} and describe an algorithm for applying MLMC with respect to time $T$ for the true solution of \eqref{eq:langevin}. In Section \ref{sec:mlmc} we present the new algorithm, as well as a framework that allows proving its convergence properties for a numerical method of choice. In Section \ref{sec:ex}  we present two examples of suitable numerical methods, while in Section \ref{sec:sgld} we describe a new version of SGLD with complexity $\mathcal{O}(\varepsilon^{-2}\left|\log\varepsilon\right|^{3})$. We conclude in Section \ref{sec:num} where a number of relevant numerical experiments are described.

\section{Preliminaries}
\label{sec:prel}
In Section \ref{subsec:mlmc} we review the classic, finite time, MLMC framework, while in Section \ref{subsec:contract}  we state the key  asymptotic properties of solutions of \eqref{eq:langevin} when $U$ is strongly concave. 

\subsection{MLMC with fixed terminal time.}
\label{subsec:mlmc}
Fix $T>0$ and consider the problem of approximating $\E[g(X_T)]$  where $X_T$ is a solution of the SDE \eqref{eq:langevin} and $g:\RR^d\rightarrow \RR$.  A classical approach to this problem consists of constructing a biased (bias arising due to time-discretisation) estimator of the form 
\begin{equation} \label{eq:MC}
\frac{1}{N}\sum_{i=1}^N g((x_T^{M})^{(i)}) \,,
\end{equation}
where $(x_T^{M})^{(i)}$ for $i = 1, \ldots , N$ are independent copies of the random variable $x_T^M$, with $\{ x_{kh}^M \}_{k=0}^{M}$ being a discrete time approximation of \eqref{eq:langevin} over $[0,T]$ with the discretisation parameter $h$ and with $M$ time steps, i.e., $Mh = T$. A central limit theorem for the estimator \eqref{eq:MC} has been derived in  \cite{MR1384358}, and it was  shown that
its computational complexity is  $\mathcal{O}(\varepsilon^{-3})$, 
for the root mean square error $\mathcal{O}(\varepsilon)$ (as opposed to $\mathcal{O}(\varepsilon^{-2})$ that can be obtained if we could sample $X_{T}$ without the bias). 
The recently developed MLMC approach allows recovering optimal  complexity $\mathcal{O}(\varepsilon^{-2})$, despite the fact that the estimator used therein builds on biased samples. This is achieved by exploiting the following identity \cite{Giles2015Acta,MR2187308}
\begin{equation} \label{eq:telescoping}
\E[g_L]	= \E[g_0] + \sum_{\ell=1}^{L} \E[ g_{\ell} - g_{\ell-1} ],
\end{equation}
where $g_\ell:=g(x_T^{M_\ell})$ and for any $\ell=0\ldots L$ the Markov chain $\{x_{kh_{\ell}}^{M_\ell}\}_{k=0}^{M_{\ell}}$ is the discrete time approximation of \eqref{eq:langevin} over $[0,T]$, with the discretisation parameter $h_\ell$ and with $M_{\ell}$ time steps (hence $M_\ell h_\ell= T$). This identity leads to an unbiased estimator of $\E[g_L]$ given by
 \begin{align*} 
 \frac{1}{N_0} \sum_{i=1}^{N_0} g_0^{(i,0)} + \sum_{\ell=1}^{L}\left\{ \frac{1}{N_\ell} \sum_{i=1}^{N_\ell} 
 (    g_{\ell}^{(i,\ell)} - g_{\ell-1}^{(i,\ell)} )  \right\},	
\end{align*}
where $g_{\ell}^{(i,\ell)}= g((x_T^{M_\ell})^{(i)})$ and $g_{\ell - 1}^{(i,\ell)}= g((x_T^{M_{\ell - 1}})^{(i)})$ are independent samples at level $\ell$.  The inclusion of the level $\ell$ in the superscript $(i,\ell)$ indicates that independent samples are used at each level $\ell$. The efficiency of MLMC lies in the coupling of $g_{\ell}^{(i,\ell)}$ and $g_{\ell-1}^{(i,\ell)}$ that results in small $\mathrm{Var}[g_{\ell} - g_{\ell-1} ]$. In particular, for the  SDE \eqref{eq:langevin} one can use the same Brownian path to simulate $g_{\ell}$ and $g_{\ell-1}$ which, through the  strong convergence property of the underlying numerical  scheme used, yields an estimate for $\mathrm{Var}[g_{\ell} - g_{\ell-1} ]$. 

By solving a constrained optimization problem (cost \&accuracy) one can see that reduced computational complexity (variance) arises 
since the MLMC method allows one to efficiently combine many simulations on low accuracy grids (at a corresponding low cost), with relatively few simulations computed with high accuracy and high cost on very fine grids.  It is shown in Giles \cite{Giles2015Acta} that under the 
assumptions\footnote{Recall $h_{\ell}$ is the time step used in the discretization of the level $\ell$.} 
\begin{eqnarray}
\label{ml_ass}
\bigl|\E[g- g_{\ell} ]|\leq c_1h_\ell^{\alpha},\quad \mathrm{Var}[g_{\ell} - g_{\ell-1} ]\leq c_2 h_\ell^{\beta}, 
\end{eqnarray}
for some \(\alpha\geq 1/2,\) \(\beta>0,\) \(c_1>0\) and \(c_2>0,\) the computational complexity of the resulting multi-level estimator with  accuracy \(\varepsilon\)  is proportional to 
\begin{eqnarray*}
\mathcal{C}=
\begin{cases}
\varepsilon^{-2}, & \beta>\gamma, \\
\varepsilon^{-2}\log^2(\varepsilon), & \beta=\gamma, \\
\varepsilon^{-2-(1-\beta)/\alpha}, & 0<\beta <\gamma
\end{cases}
\end{eqnarray*}
where the cost of the algorithm is of order $h_L^{-\gamma}$. Typically,  the constants $c_{1},c_{2}$ grow exponentially in time $T$ as they follow from classical finite time weak and strong convergence analysis of the numerical schemes. The aim of this paper is to establish the bounds \eqref{ml_ass} uniformly in time, i.e., to find constants $\widetilde{c}_1$, $\widetilde{c}_2 > 0$ independent of $T$ such that
\begin{eqnarray} \label{uml_ass}
\sup_{T>0}\bigl|\E[g - g_{\ell} ]|\leq \widetilde{c}_1 h_\ell^{\alpha},\quad 
\sup_{T>0}\mathrm{Var}[g_{\ell} - g_{\ell-1} ]\leq \widetilde{c}_2 h_\ell^{\beta}. 
\end{eqnarray}

\begin{remark}
The reader may notice that in the regime when $\beta>\gamma$, the computationally complexity of $\mathcal{O}(\varepsilon^{-2})$  coincides with that of an unbiased estimator. Nevertheless, the MLMC 
estimator as defined here is still biased, with the bias being  controlled by the choice of final  level parameter $L$. However,  in this setting it is possible to eliminate the bias by a clever randomisation trick \cite{Rhee2015}. 
\end{remark}

\subsection{Properties of ergodic SDEs with strongly concave drifts}
\label{subsec:contract}

Consider the SDE \eqref{eq:langevin} and let $U$ satisfy the following condition
\begin{itemize}
\item[]\textbf{HU0}  For any $x,y \in \R^d$ there exists a positive constant $m$ s.t.\
\begin{eqnarray} \label{eq:ing}
\left\langle \nabla U(y) - \nabla U(x),y-x\right\rangle  & \leq - m | x-y| ^{2},
\end{eqnarray}
\end{itemize}
which is also known as a one-side Lipschitz condition. Condition \textbf{HU0}
is satisfied for strongly concave potential, i.e., when for any $x,y \in \R^d$ there exists constant $m$ s.t.\		
\begin{equation*}
U(y)  \leq  U(x)+\left\langle \nabla U(x),y-x\right\rangle - \frac{m}{2}| x-y| ^{2}.
\end{equation*}
In addition  \textbf{HU0} implies that 
\begin{eqnarray} \label{eq:ing}
\left\langle   \nabla U(x),x\right\rangle  & \leq - \frac{m}{2} | x| ^{2}  + \frac{1}{2m} | \nabla U(0)|^2, \quad \forall x\in \RR^d
\end{eqnarray}
which in turn implies that 
\begin{eqnarray} \label{eq:ing1} 
\left\langle \nabla U(x), x \right\rangle  & \leq - m' | x| ^{2} + 2b | \nabla U(0) |^2, \quad \forall x\in \RR^d 
\end{eqnarray}
for some\footnote{If $\nabla U(0)=0$ then $m'=m, b=0$. Otherwise $m'<m$ (implication of Young's inequality).}
 $m'>0,b\geq 0$.
Condition \textbf{HU0} ensures the contraction needed to establish uniform in time
estimates for the solutions of \eqref{eq:langevin}. For the transparency of the exposition we introduce  the following flow notation for the solution to
\eqref{eq:langevin}, starting at $X_{0}=x$
\begin{equation} \label{eq:flow}
\psi_{s,t,W}(x) := x + \int_s^t\nabla U(X_{r})dr + \int_s^t \sqrt{2}dW_{r},\quad x\in\IR^{d}.
\end{equation}
The theorem below demonstrates that solutions to \eqref{eq:langevin}  driven by the same Brownian motion, but with different initial conditions  enjoy  an  exponential contraction property. 
\begin{theorem} \label{thm:contract}
Let $(W_t)_{t\geq 0}$ be a standard Brownian Motion in $\RR^d$. We fix random variables $Y_0$, $X_0\in \IR^{d}$ and define $X_T=\psi_{0,T,W}(X_0)$ and $Y_T=\psi_{0,T,W}(Y_0)$. If \textbf{HU0} holds, then 
\begin{equation} \label{eq:contract_true}
\IE |X_{T}-Y_{T}|^{2} \leq  \IE |X_{0}-Y_{0}|^{2} e^{-2mT} 
\end{equation}
\end{theorem}
\begin{proof}
The result follows from It\^o's formula.  Indeed we have

\[
\frac{1}{2}e^{2mt} |X_{t}-Y_{t}|^{2}= \frac{1}{2}|X_{0}-Y_{0}|^{2}+\int_{0}^{t}m e^{2ms } |X_{s}-Y_{s}|^{2}ds 
 + \int_{0}^{t} e^{2ms}\left\langle \nabla U(X_{s}) - \nabla U(Y_{s}),X_{s}-Y_{s}\right\rangle ds.
\]
Assumption \textbf{HU0} yields 
\[
 \E|X_{T}-Y_{T}|^{2} \leq e^{-2mT}\IE |X_{0}-Y_{0}|^{2},
\]
as required.
\end{proof}

\begin{remark} The $2$-Wasserstein distance between probability measures $\nu_{1}$
and $\nu_{2}$ defined on a Polish metric space $E$, is given by 
\begin{eqnarray*}
\mathcal{W}_{2}(\nu_{1},\nu_{2}) & = & \left(\inf_{\pi\in\Gamma(\nu_{1},\nu_{2})}\int_{E\times E}|x-y|^{2}\pi(dx,dy)\right)^{\frac{1}{2}},
\end{eqnarray*}
with $\Gamma(\nu_{1},\nu_{2})$ being the set of couplings of $\nu_{1}$
and $\nu_{2}$ (all probability measures on $E\times E$ with marginals $\nu_{1}$
and $\nu_{2}$).  We denote $\mathcal{L}(\psi_{0,t,W}(x))= P_{t}(x,\cdot)$. That is $P_{t}$ is the transition kernel of the SDE \eqref{eq:langevin}.  
Since the choice of the same driving Brownian Motion in Theorem \ref{thm:contract} is an example of a coupling, equation \eqref{eq:contract_true} implies 
\begin{equation}
\mathcal{W}_{2}\left(P_{t}(x,\cdot),P_{t}(y,\cdot)\right)\leq| x-y| \exp\left(-mt\right)\label{eq:wassersteincontracting}
\end{equation}
Consequently $P_{t}$ has a unique invariant measure
and thus the process is ergodic \cite{weakHarris}.
In the present paper we are not concerned with determining couplings that are optimal; for  practical considerations one should only consider couplings that are feasible to implement (see also discussion in 
\cite{Agapiou2014unbiasing,giles2014antithetic}).
\end{remark}

\subsection{Coupling in time $T$}

For the MLMC method with different discretisation parameters on different levels, coupling with the same Brownian motion  is not enough to obtain good upper bounds on the variance, as in general solutions to SDEs \eqref{eq:langevin} are $1/2$-H\"{o}lder continuous,  \cite{krylov2008controlled}, i.e., for any $t>s>0$ there exists a constant $C>0$ such that
\begin{equation} \label{eq:holder}
	\E |X_t - X_s |^2 \le C |t-s|
\end{equation}
and it is well known that this bound is sharp. As we shall see later this bound will not lead to an efficient MLMC implementation.  However, by suitable coupling of the SDE solutions on time intervals of length $T$ and $S$, $T > S$, respectively,  we will be able to take advantage of the exponential contraction property obtained in Theorem \ref{thm:contract}. 

\begin{figure}[htb]
\centering
\subfloat[Correct coupling.]{
\includegraphics[scale=.35]{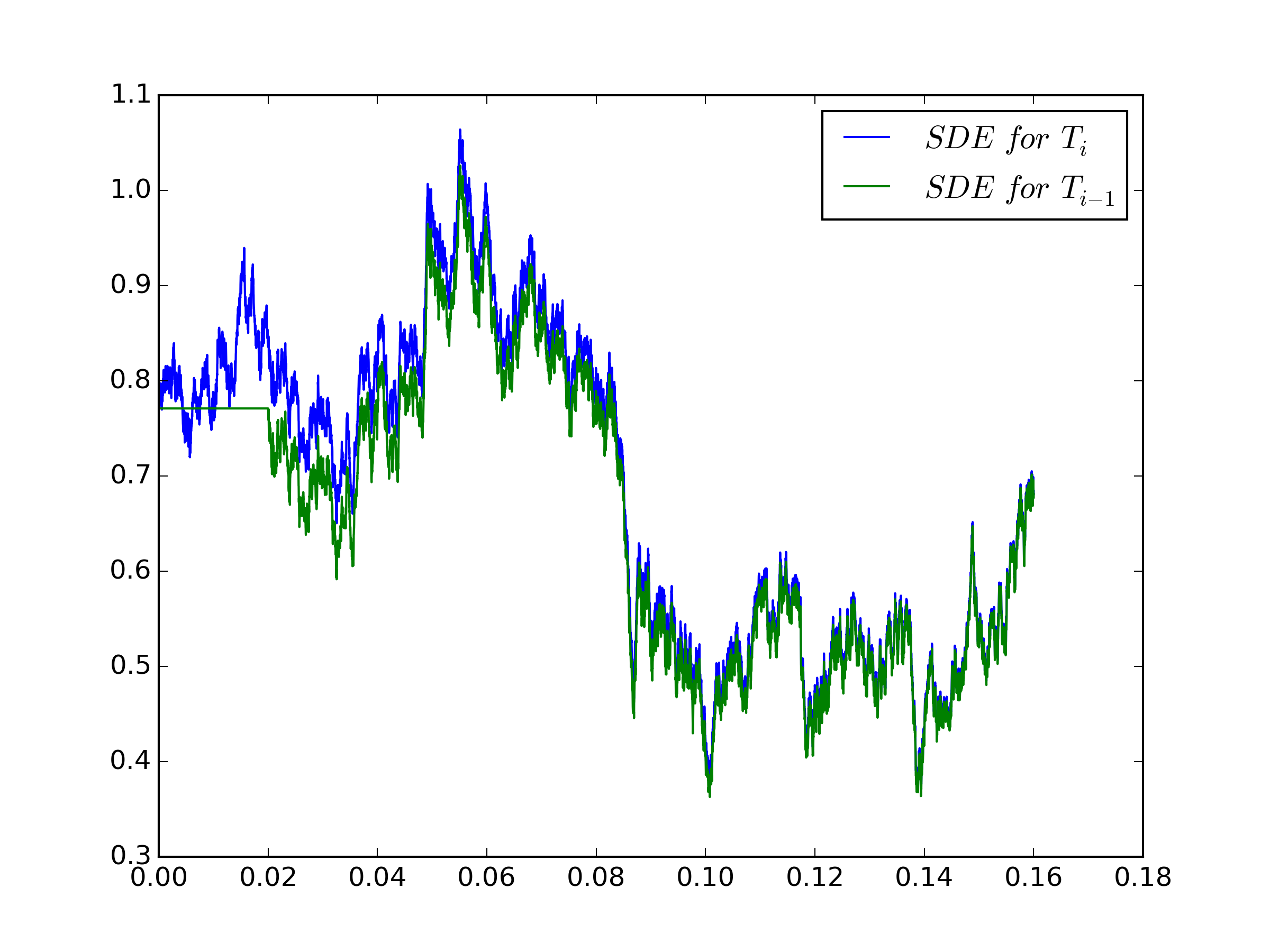}
} 
\subfloat[Wrong coupling.]{
\includegraphics[scale=.35]{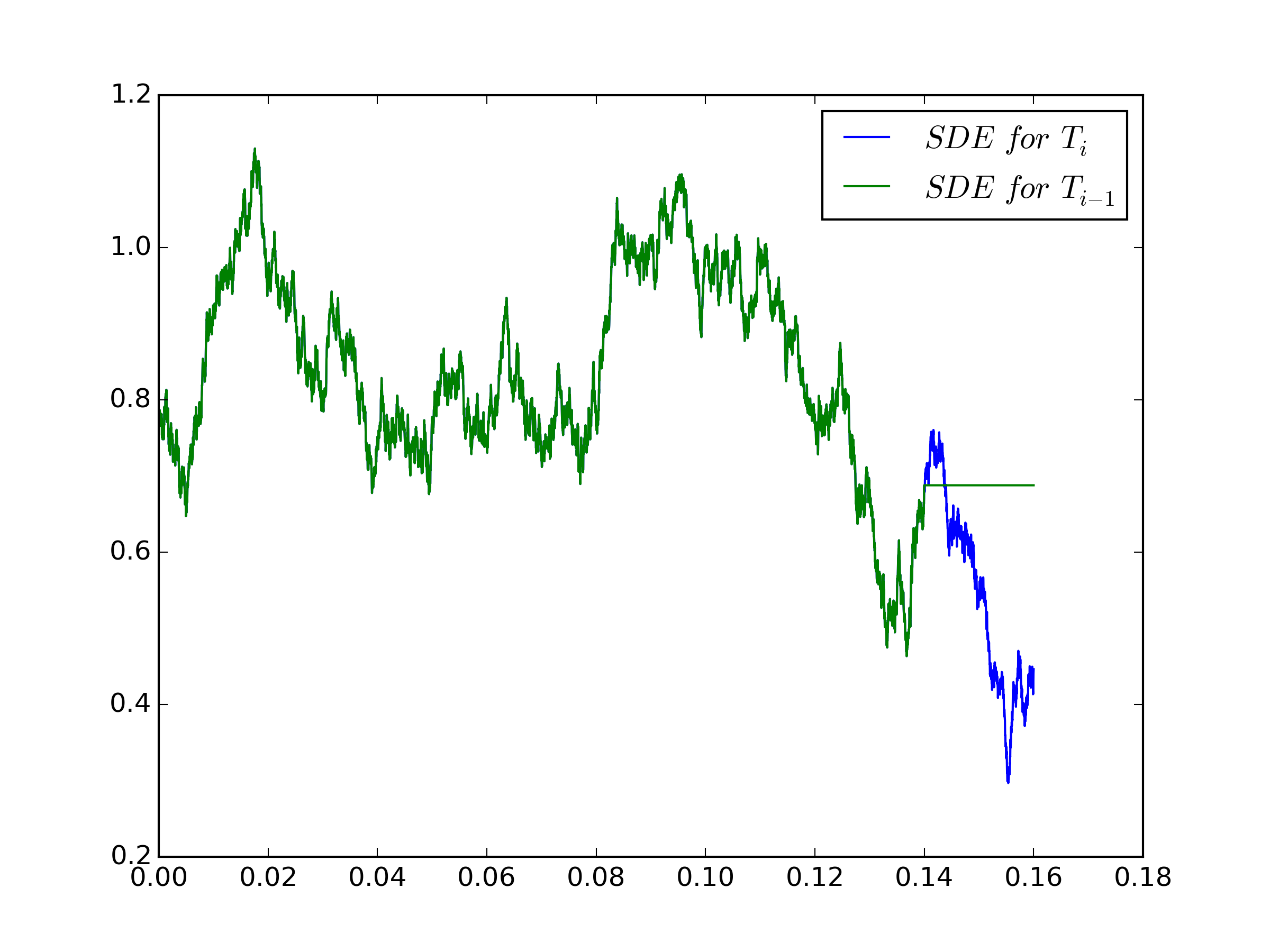}
}
\caption{Shifted couplings}
\label{fig:couplings}
\end{figure}

Let $(T_\ell)_{\ell \leq 0}$ be an increasing sequence of positive real numbers. To couple processes with different terminal times $T_i$ and $T_j$, $i\neq j$, we exploit the time homogeneous Markov property of the flow \eqref{eq:flow}. More precisely, for each $\ell \geq 0$ one would like to   construct 
 a pair $(\mathcal{X}^{(f,\ell)},\mathcal{X}^{(c,\ell)})$ of solutions to \eqref{eq:langevin}, which we refer to as fine and coarse paths, such that 
\begin{equation}\label{eq:law}
\begin{split}
	\mathcal{L}(\mathcal{X}^{(f,\ell)}(T_{\ell - 1}))&=\mathcal{L}(X_{T_{\ell}}), \\ \mathcal{L}(\mathcal{X}^{(c,\ell)}(T_{\ell - 1}))&=\mathcal{L}(X_{T_{\ell-1}}), \quad \forall \ell \geq 0,
	\end{split}
\end{equation}
and 
\begin{equation} \label{eq:variance_red}
\E |\mathcal{X}^{(f,\ell)}(T_{\ell - 1}) - \mathcal{X}^{(c,\ell)}(T_{\ell - 1}) |^2  \leq \E |X_{T_\ell} - X_{T_{\ell-1}} |^2. 	
\end{equation}
Following \cite{Rhee2012,Rhee2015, Agapiou2014unbiasing,Giles2015Acta} we propose a particular coupling denoted by $(X^{(f,\ell)},X^{(c,\ell)})$, and constructed in the following way (see also Figure \ref{fig:couplings}a)
\begin{itemize}
\item First \footnote{As we can see in Figure \ref{fig:couplings}b, doing this first is important for the overall difference of the paths.} obtain a solution to \eqref{eq:langevin} over  $[0,T_\ell - T_{\ell-1}]$.
We fix an $\mathbb{R}^d$-valued random variable $X(0)$ and take  $X^{(f,\ell)}(0) = \psi_{0,(T_\ell - T_{\ell-1}),\tilde{W}}(X(0)) $. 
\item Next  couple fine and coarse paths on the remaining time interval $[0, T_{\ell-1}]$ using the same Brownian motion $W$ i.e.,
\[
X^{(f,\ell)}(T_{\ell-1}) = \psi_{0,T_{\ell-1},W}(X^{(f,\ell)}(0)), 
\]
and
\[
 X^{(c,\ell)}(T_{\ell-1}) = \psi_{0,T_{\ell-1},W}(X(0)).
\]  
\end{itemize}
We note here that $\nabla U(\cdot)$ in \eqref{eq:langevin} is time homogenous, hence the same applies for the  corresponding transition kernel $\mathcal{L}(\psi_{0,t,W}(x))= P_{t}(x,\cdot)$, which implies that condition 
\eqref{eq:law}  holds.  Now
  Theorem \ref{thm:contract} yields
{\small
\begin{equation} \label{eq:sdetime}
\IE|X^{(f,\ell)}(T_{\ell-1}) - X^{(c,\ell)}(T_{\ell-1})|^{2} \leq \IE|X^{(f,\ell)}(0)-X(0)|^{2} e^{-2mT_{\ell-1}}.
\end{equation}}implying that condition \eqref{eq:variance_red} is also satisfied. We now take $\rho>1$ and define 
\begin{equation} \label{eq:T_choice_true}
T_{\ell} := \frac{\log{2}}{2m} \rho (\ell+1) \quad \forall \ell\geq 0.
\end{equation}
In our case   $g_{\ell}^{(i,\ell)}=g((X^{(f,\ell)}(T_{\ell-1}))^{(i)})$ and   $g_{\ell-1}^{(i,\ell)} =g((X^{(c,\ell)}(T_{\ell-1}))^{(i)})$ and we assume that $g$ is globally Lipschitz with Lipschitz constant $K$. Hence 
\begin{align*} 
\IE|g(X^{(f,\ell)}(T_{\ell-1})) - g(X^{(c,\ell)}(T_{\ell-1}))|^{2}	  \leq & K^{2}\IE|X^{(f,\ell)}(T_{\ell-1}) - X^{(c,\ell)}(T_{\ell-1})|^{2} \\
 \leq & K^{2}\IE|X^{(f,\ell)}(0)-X(0)|^{2} e^{-2mT_{\ell-1}}\\
\leq &  K^{2}\IE|X^{(f,\ell)}(0)-X_{0}|^{2} 2^{-\rho \ell} \\
\leq &  K^{2}C | T_{\ell}-T_{\ell-1} |2^{-\rho \ell}.
\quad \forall i\geq 0, \label{eq:variance_true_decay}
\end{align*}
where the last inequality follows from \eqref{eq:holder}. 
\section{MLMC in $T$ for approximation of SDEs}
\label{sec:mlmc}
Having described a coupling algorithm with good contraction properties, we now present the main algorithm in Section \ref{subsec:alg_descr}, while in Section \ref{subsec:conv_anal} we present a general numerical analysis framework for proving the convergence of our algorithm.

\subsection{Description of the general algorithm}
\label{subsec:alg_descr}
We now focus on the  numerical discretisation of the Langevin equation \eqref{eq:langevin}.  In particular, we are interested in coupling the 
discretisations of  \eqref{eq:langevin}
based on step size $h_{\ell}$ and $h_{\ell-1}$ with $h_{\ell}=h_{0}2^{-\ell}$. Furthermore, as we are interested in computing expectations with respect to the invariant measure $\pi(dx)$ we also increase the time endpoint $T_{\ell}\uparrow\infty$
which is chosen such that $\frac{T_{\ell}}{h_{0}}, \frac{T_{\ell}}{h_{\ell}}\in\mathbb{N}$. We illustrate the  main idea using two generic discrete time stochastic processes $(x_k)_{k\in \mathbb{N}},(y_k)_{k\in \mathbb{N}}$ which can be defined as
\begin{equation} \label{eq:genpro}
x_{k+1}= S^f_{h,\xi_k}(x_k), \quad  y_{k+1}=S^c_{h,\tilde{\xi}_k}(y_k),	
\end{equation}
where $S_{h,\xi_k}(x_k) = S(x_k,h,\xi_k)$ and the operators $S^{f},S^{c}: \RR^d \times \RR_+ \times \RR^{d\times m} \rightarrow \RR^d $ are Borel measurable, whereas $\xi,\tilde{\xi}$ are random inputs to the
algorithms. The operators 
$S^f$ and $S^c$ in \eqref{eq:genpro}
need not  be the same. This extra flexibility allows analysing various coupling ideas. 

For example for 
the Euler discretisation we have
\[
S_{h,\xi}(x)=x+h\nabla U(x)+\sqrt{2h} \xi,
\]
where $\xi\sim\mathcal{N}(0,I)$. We will also use the notation $P_{h}(x,\cdot)=\mathcal{L}\left(S_{h,\xi}(x)\right)$ for the corresponding Markov kernel. 

For MLMC algorithms one evolves both fine and coarse paths  jointly, over a time interval of length $T_{\ell-1}$, by doing two steps for the
finer level (with the time step $h_\ell$) and one on the coarser level (with the time step $h_{\ell-1}$). We will use the notation $(x^f_{\cdot}),(x^c_{\cdot})$ for
\begin{align}
x_{k+\frac{1}{2}}^{f} & =S_{\frac{h}{2},\xi_{k+\frac{1}{2}}}^{f}\left(x_{k}^{f}\right), \quad x_{k+1}^{f}  =S_{\frac{h}{2},\xi_{k+1}}^{f}\left(x_{k+\frac{1}{2}}^{f}\right) \label{eq:numOperator}\\
x_{k+1}^{c} & =S_{h,\tilde{\xi}_{k+1}}^{c}(x_{k}^{c}) \label{eq:numOperator2}.
\end{align}
The algorithm generating $(\xf{f}{k})_{k\in \mathbb{N}/2}$ and $(\yc{c}{k})_{k\in \mathbb{N}}$ is presented in  Algorithm \ref{alg:CouplingLangevinDiscretisation}.

\begin{algorithm}
\begin{enumerate}
\item Set $x_{0}^{(f,\ell)}=x_{0}$, then simulate according
to $P_{h_{\ell}}(x_0,\cdot)$ up~to time $\frac{T_{\ell}-T_{\ell-1}}{h_{\ell}}$, thus obtaining $x_{\frac{T_{\ell}-T_{\ell-1}}{h_{\ell}}}^{(f,\ell)}$; 
\item Set $x_{0}^{(c,\ell)}=x_{0}$ and $x_{0}^{(f,\ell)}=x_{\frac{T_{\ell}-T_{\ell-1}}{h_{\ell}}}^{(f,\ell)}$,
then simulate $(x_{\cdot}^{(f,\ell)},x_{\cdot}^{(c,\ell)})$ jointly
as 
\begin{align*}
\textstyle
\left(x_{k+1}^{(f,\ell)},x_{k+1}^{(c,\ell)}\right)  = \Big( &S^{f}_{h_{\ell},\xi_{k + 1}}\circ S^{f}_{h_{\ell},\xi_{k + \frac{1}{2}}}(x_{k}^{(f,\ell)}), 
 S^{c}_{h_{\ell-1},\frac{1}{\sqrt{2}}\left(\xi_{k + \frac{1}{2}} + \xi_{k + 1}\right)}(x_{k}^{(c,\ell)})\Big).
\end{align*}
\item Set $k_\ell := \frac{T_{\ell-1}}{h_{\ell-1}}$ and
\[
\Delta^{(i)}_{\ell}:=g\left( \left(x_{k_\ell}^{(f,\ell)} \right)^{(i)}\right)-g\left(\left(x_{k_\ell}^{(c,\ell)}\right)^{(i)}\right)
\]
\end{enumerate}
\caption{\label{alg:CouplingLangevinDiscretisation}Coupling Langevin discretisations
for $T_{\ell}\uparrow \infty$.}
\end{algorithm}

\subsection{General convergence analysis}
\label{subsec:conv_anal}
We will now present a general theorem for estimating the bias and the variance in the 
MLMC set up.  We refrain from prescribing the exact dynamics of $(x_k)_{k\geq 0}$ and $(y_k)_{k\geq 0}$ in \eqref{eq:genpro}, as we seek 
general conditions allowing the  construction of uniform in time 
approximations of \eqref{eq:langevin} in the $L^2$-Wasserstein norm. The advantage of working in this general setting is that if we wish to work with 
more advanced numerical schemes than the Euler method (e.g.\ implicit, projected, adapted or randomised scheme) or general noise terms (e.g.\ $\alpha$-stable processes), it will be sufficient to verify relatively simple 
conditions to see the 
performance of the complete algorithm. To give the reader some intuition behind the abstract assumptions, we discuss the specific methods in Section \ref{sec:ex}.

\subsubsection{Uniform estimates in time}

\begin{definition}[Bias]
\label{def:bias}
We say that a process $(x_k)_{k\in \mathbb{N}}$ converges weakly uniformly in time with order $\alpha>0$ to the solution $(X_t)_{t \geq 0}$ of the SDE 
\eqref{eq:langevin}, if there exists a constant $c>0$ such that for any $h > 0$,
\begin{equation*}
\sup_{t\geq 0}|\mathbb{E}[g(X_t)] - \E[g(x_{\lfloor t/h \rfloor})]|\leq c h^{\alpha}\, \quad g\in C^r_K(\mathbb{R}) \,.	
\end{equation*} 
\end{definition}
 We define MLMC variance as follows.

\begin{definition}[MLMC variance]
Let the operators in  \eqref{eq:numOperator}-\eqref{eq:numOperator2}
satisfy that for all $x$
\begin{align}
 & \mathcal{L}\left(S^{f}_{h,\xi}(x)\right)=\mathcal{L}\left(S_{h,\tilde{\xi}}^{c}(x)\right).\label{eq:W1}
\end{align}
We say that the MLMC variance is of order $\beta>0$ if there exists a constant $c_V>0$ s.t.\ for any $h > 0$, 
\begin{equation} \label{eq:W2}
\sup_{t\geq 0}\mathbb{E}|g(x^{c}_{\lfloor t/h\rfloor}) - g(x^{f}_{\lfloor t/h\rfloor})|^2 \leq c_V h^{\beta}.
\end{equation}
\end{definition}

\subsubsection{Statement of sufficient conditions}

We now discuss the necessary conditions imposed on a generic numerical method \eqref{eq:genpro} to estimate MLMC variance. We decompose the global error into the one step error and the regularity of the scheme. To proceed we introduce the notation $x^h_{k,x_s} $ for the process at time $k$
with initial condition $x_s$ at time $s<k$. If it is clear from the context what initial condition is used we just write $x^h_{k}$. We also define the conditional expectation operator as  
$\E_n[\cdot]:=\E[\cdot | \mathcal{F}_{n}]$, where $\mathcal{F}_n := \sigma \left( x^h_k : k \leq n \right)$ . 

We now have the following definition.

\begin{definition}[$L^2$- regularity] \label{def:reg2}
We will say that the one step operator $S : \RR^d \times \RR_+ \times \RR^{d\times m} \rightarrow \RR^d $  is $L^2$-regular \textbf{uniformly in time} if for any $\cF_n$-measurable random variables $x_n$, $y_n\in \RR^d$ there exist 
constants $K$, $C_{\mathcal{R}}$, $C_{\mathcal{H}}$, $\widetilde{\beta}>0$ and random variables $Z_{n+1}$, $\mathcal{R}_{n+1} \in \cF_{n+1}$
and $\mathcal{H}_{n} \in \cF_{n}$,
such that for all $h \in (0,1)$ 
\begin{align*}
S_{h,\xi_{n+1}}(x_n) - S_{h,\xi_{n+1}}(y_n) = x_n - y_n + Z_{n+1} 
\end{align*}
and
\begin{align}
	\E_n[ |S_{h,\xi_{n+1}}(x_n) - S_{h,\xi_{n+1}}(y_n) |^2] & \le (1 - Kh)|x_n - y_n |^2  + \mathcal{R}_{n} \nonumber \\
	\E_n[|Z_{n+1}|^2] &\le   \mathcal{H}_{n}   |x_n - y_n |^2 h, \label{e:conditionZn}
\end{align}
where 
\begin{equation}
\begin{split}
 \label{eq:cond_sum2}
\sup_{n\geq 1}\E\left[  \sum_{i=0}^{n-1} e^{(i-(n-1))hK/2} \mathcal{R}_i \right] &\le C_{\mathcal{R}} h^{\widetilde{\beta}}, \\
	\sup_{n\geq 1}\E\left[ |\mathcal{H}_n|^2\right] &\le C_{\mathcal{H}}.
\end{split}  
\end{equation} 
\end{definition}
We now introduce the set of the assumptions needed for the proof of the main convergence theorem.

\begin{assumption}
Consider two processes $(\xf{f}{k})_{2k\in \mathbb{N}}$ and $(\yc{c}{k})_{k\in \mathbb{N}}$ obtained from the recursive application of the the operators $S^f_{h,\xi}(\cdot)$ and $S^c_{h,\xi}(\cdot)$ as defined in \eqref{eq:genpro}. 
We assume that
\begin{itemize}
\item[\textbf{H0}] There exists a constant $H>0$ such that for all $q>1$
\begin{align*}
\sup_k \E |\xf{f}{k}|^q \leq H_ \quad \text{and} \quad \sup_k \E |\yc{c}{k}|^q  \leq H\,.  
\end{align*}
\item[\textbf{H1}] For any $x\in\RR^d$ 
\[
\mathcal{L}\left(S^{f}_{h,\xi}(x)\right)=\mathcal{L}\left(S_{h,\tilde{\xi}}^{c}(x)\right).
\]
\item[\textbf{H2}]  The  operator $S^f_{h,\xi}(\cdot)$ is $L^2$-regular uniformly in time. 
\end{itemize}
\end{assumption}


Below we present the main convergence result of this section. By analogy to \eqref{eq:numOperator}-\eqref{eq:numOperator2}, we use the notation
\begin{equation*}
\begin{split}
\xf{f}{{n,\yc{c}{{n-1}}}} &= S^f_{h,\xi_{n}}\left(\xf{f}{{n-\frac{1}{2},\yc{c}{{n-1}}}} \right) \\
 \xf{f}{{n-\frac{1}{2},\yc{c}{{n-1}}}} &= S^f_{h,\xi_{n-\frac{1}{2}}}\left(\yc{c}{{n-1}} \right) \\
 \yc{c}{{n,\yc{c}{{n-1}}}} &= S^c_{h,\tilde{\xi}_n}\left( \yc{c}{{n-1}} \right) \,.
\end{split}
\end{equation*}

Using the estimates derived here we can immediately estimate the rate of decay of MLMC variance. 

\begin{theorem} \label{th:convergence}
Take  $(\xf{f}{n})_{n\in \mathbb{N}/2}$ and $(\yc{c}{n})_{n\in \mathbb{N}}$ with $h\in(0,1]$ and assume that  
\textbf{H0}-\textbf{H2} hold. Moreover, assume that there exist constants $c_{s}>0,c_{w}>0$  and 
$\alpha \geq \frac{1}{2}$, $\beta \geq 0$, $p \geq 1$ with $\alpha \geq \frac{\beta}{2}$  such that for all  $n\geq 1$
\begin{equation} \label{eq:weak}
|\E_{n-1}(\yc{c}{{n,\yc{c}{{n-1}}}} - \xf{f}{{n,\yc{c}{{n-1}}}})| \leq c_{w}(1+ |\yc{c}{{n-1}}|^p)h^{\alpha+1},
\end{equation}
and
\begin{equation} \label{eq:strong}
\E_{n-1}[|\yc{c}{{n,\yc{c}{{n-1}}}} - \xf{f}{{n,\yc{c}{{n-1}}}} |^2] \leq c_{s}(1+|\yc{c}{{n-1}}|^{2p})h^{\beta+1}.
\end{equation}
Then the global error is bounded by
\begin{align*}
 \E[|\yc{c}{{T/h,x_{0}}} - \xf{f}{{T/h,y_0}} |^2]  &\leq   |x_0-y_0 |^2 e^{- K/2 T} + 2\Gamma h^{\beta}/K \\ & +  \sum_{j=0}^{n-1}e^{ (j-(n-1))h K /2 }\E(\mathcal{R}_{j})\,,
 \end{align*}
where $T/h = n$ and $\Gamma$ is given by \eqref{eq gammma2}.
\end{theorem}

\begin{proof}
We begin using the following identity
\begin{align*}
	\yc{c}{{n,x_{0}}} - \xf{f}{{n,y_{0}}} =& \yc{c}{{n, \yc{c}{{n-1}}    } } -  \xf{f}{{n, \xf{f}{{n-1}}}     }  \\
	 =& (\yc{c}{{n, \yc{c}{{n-1}}    } } -  \xf{f}{{n, \yc{c}{{n-1}}    } } ) +   ( \xf{f}{{n, \yc{c}{{n-1}}    } } -   \xf{f}{{n, \xf{f}{{n-1}}    } }).
\end{align*}
We will be able to deal with the first term in the sum by using equations \eqref{eq:weak} and \eqref{eq:strong}, while the second term will be controlled because of the $L^{2}$ regularity of the numerical scheme. Indeed, by squaring both sides in the equality above we have 
\[
	 |\yc{c}{{n,y_{0}}} - \xf{f}{{n,x_{0}}}|^2  = 
	  | \yc{c}{{n, \yc{c}{{n-1}}    } } -  \xf{f}{{n, \yc{c}{{n-1}}    } } |^2 
	   +  |\xf{f}{{n, \yc{c}{{n-1}}    } } -   \xf{f}{{n, \xf{f}{{n-1}}    } }|^2  
	 +2 \langle \yc{c}{{n, \yc{c}{{n-1}}    } } -  \xf{f}{{n, \yc{c}{{n-1}}    } } 
	  , \yc{c}{{n-1}}  - \xf{f}{{n-1}}+ Z_{n}\rangle \,,
\]
where in the last line we have used Assumption \textbf{H2}. Applying conditional expectation operator to both sides of the above equality we obtain
\begin{align*}
	\E_{n-1}[|\yc{c}{{n,y_{0}}} - \xf{f}{{n,x_{0}}}|^2]   &= 
	  \E_{n-1}[|\yc{c}{{n, \yc{c}{{n-1}}    } } -  \xf{f}{{n, \yc{c}{{n-1}}    } } |^2] 
	+ \E_{n-1}[| \xf{f}{{n, \yc{c}{{n-1}}    } } -   \xf{f}{{n, \xf{f}{{n-1}}    } }|^2]  \\
	  &\qquad + 2  \langle \yc{c}{{n-1}}  -\xf{f}{{n-1}}, \E_{n-1}[\yc{c}{{n, \yc{c}{{n-1}}    } }  -  \xf{f}{{n, \yc{c}{{n-1}}    } } ] \rangle \\
	  &\qquad+2  \E_{n-1}  \langle Z_{n}, \yc{c}{{n, \yc{c}{{n-1}}    } } -  \xf{f}{{n, \yc{c}{{n-1}}    } } \rangle \,.
	 \end{align*}
Applying Cauchy-Schwarz inequality, and using the weak error estimate \eqref{eq:weak} leads to
\begin{align*}
	\E_{n-1}[|\yc{c}{{n,y_{0}}} - \xf{f}{{n,x_{0}}}|^2]  & \leq 
	  \E_{n-1}[|\yc{c}{{n, \yc{c}{{n-1}}    } } -  \xf{f}{{n, \yc{c}{{n-1}}    } } |^2] 
	+ \E_{n-1}[| \xf{f}{{n, \yc{c}{{n-1}}    } } -   \xf{f}{{n, \xf{f}{{n-1}}    } }|^2]  \\
	  &\qquad + 2   c_{w} h^{\alpha+1} | \yc{c}{{n-1}}  -\xf{f}{{n-1}}| (1+ |\yc{c}{{n-1}}|^p) \\
	  &\qquad + 2 (\E_{n-1}[ |Z_{n}|^2])^{1/2} (\E_{n-1}[| \yc{c}{{n, \yc{c}{{n-1}}    } } -  \xf{f}{{n, \yc{c}{{n-1}}    } } |^2])^{1/2}. 
	 \end{align*}
By assumptions \textbf{H0}-\textbf{H2}, and the strong error estimate \eqref{eq:strong}  we have
{\small
\begin{align*}
	 \E_{n-1}[|\yc{c}{{n,y_{0}}} - \xf{f}{{n,x_{0}}}|^2]  &\leq 
	  c_{s}(1+|\yc{c}{{n-1}}|^{2p})h^{\beta +1} 
	+ |\yc{c}{{n-1}}-\xf{f}{{n-1}}|^2 (1 - K h ) + \mathcal{R}_{n-1}\\
	  & + 2   c_{w} h^{\alpha+1} | \yc{c}{{n-1}}  -\xf{f}{{n-1}}|  (1+ |\yc{c}{{n-1}}|^p) \\
	  & + 2 \Big( \E_{n-1}[\mathcal{H}_{n}]   |\yc{c}{{n-1}}-\xf{f}{{n-1}}|^2 h \Big)^{1/2} 
	  \Big( c_{s}(1+|\yc{c}{{n-1}}|^{2p})h^{\beta +1}\Big)^{1/2} \\
	  \leq & 
	   c_{s}(1+|\yc{c}{{n-1}}|^{2p})h^{\beta +1}
	+ |\yc{c}{{n-1}}-\xf{f}{{n-1}}|^2 (1 - K h ) + \mathcal{R}_{n-1}\\
	  & + 2   c_{w} h^{\alpha+1} | \yc{c}{{n-1}}  -\xf{f}{{n-1}}|  (1+ |\yc{c}{{n-1}}|^p) \\
	  & + 2 \Big(   |\yc{c}{{n-1}}-\xf{f}{{n-1}}|^2 h \Big)^{1/2} 
	  \Big( c_{s} \E_{n-1}[\mathcal{H}_{n} (1+|\yc{c}{{n-1}}|^{2p}) ] h^{\beta +1}\Big)^{1/2} \,,
	 \end{align*}}
while taking expected values and applying Cauchy-Schwarz inequality and the fact that $\alpha \geq\frac{\beta}{2} $ and
$h<1$ (and hence $h^{\alpha+1}\leq h^{\frac{\beta}{2} +1} $)   gives 
{\small
	\begin{align*}
	 \E[|\yc{c}{{n,y_{0}}} - \xf{f}{{n,x_{0}}}|^2]  &\leq 
	  c_{s}(1+\E[|\yc{c}{{n-1}}|^{2p}])h^{\beta +1} 
	+ \E[|\yc{c}{{n-1}}-\xf{f}{{n-1}}|^2] (1 - K h ) + \E [\mathcal{R}_{n-1}]\\
	  & + 2\sqrt{2}   c_{w}  (\E[| \yc{c}{{n-1}}  -\xf{f}{{n-1}}|^2 h ])^{1/2}  (\E[(1+ |\yc{c}{{n-1}}|^{2p}) h^{\beta+1}])^{1/2} \\
	  & + 2  \E \Big[  |\yc{c}{{n-1}}-\xf{f}{{n-1}}|^2 h \Big]^{1/2}
	   \E\Big[ c_{s}\mathcal{H}_{n-1}(1+|\yc{c}{{n-1}}|^{2p})  h^{\beta+1} \Big]^{1/2}.
\end{align*} }
Now Young's inequality gives that for any $\varepsilon>0$
\begin{align*}
 \E[|\yc{c}{{n-1}}  -\xf{f}{{n-1}}|^2 h ]^{1/2} \E[(1+|\yc{c}{{n-1}}|^{2p})h^{\beta +1}]^{1/2}  
 \leq 
 \varepsilon \E[(\yc{c}{{n-1}}  -\xf{f}{{n-1}})^2]h + \frac{1}{4\varepsilon}\E[(1+|\yc{c}{{n-1}}|^{2p})]h^{\beta +1} 
 \end{align*}
 and {\small
 \begin{align*}
\E \Big[  |\yc{c}{{n-1}}-\xf{f}{{n-1}}|^2 h \Big]^{1/2} 
	  \E \Big[ c_{s}\mathcal{H}_{n-1}(1+|\yc{c}{{n-1}}|^{2p}) h^{\beta +1} \Big]^{1/2} 
 &\leq 	 \varepsilon  \E \Big[  |\yc{c}{{n-1}}-\xf{f}{{n-1}}|^2  \Big] h  \\
 &+ 
	 \frac{1}{4\varepsilon} \E \Big[ c_{s}\mathcal{H}_{n-1}(1+|\yc{c}{{n-1}}|^{2p})\Big]h^{\beta +1} \,,
\end{align*}}
while 
{\small
\[
\E \Big[ \mathcal{H}_{n-1}(1+|\yc{c}{{n-1}}|^{2p})\Big] 
\leq \frac{1}{2} \E \Big[ |\mathcal{H}_{n-1}|^2\Big] +  \E \Big[ (1+|\yc{c}{{n-1}}|^{4p})\Big]. 
\]}
Let $\gamma_n:=\E[|\yc{c}{{n,y_{0}}} - \xf{f}{{n,x_{0}}}|^2]$. Since $(1+\E[|\yc{c}{{n-1}}|^{2p}]) \leq (1+\E[|\yc{c}{{n-1}}|^{4p}])$ we have
 \begin{align*}
	\gamma_n 	  \leq 
	\left(c_{s}H+\frac{2\sqrt{2}c_{w}H+c_{s}(\E[|\mathcal{H}_{n-1}|^2]+ 2 H)}{4\varepsilon} \right) h^{\beta +1}  
	+\E[\mathcal{R}_{n-1}] 
	  +   \gamma_{n-1}  (1 -[K-(2\sqrt{2}c_{w}+2) \varepsilon] h )
	 \end{align*}
Fix  $\varepsilon=\frac{ K}{2(2\sqrt{2}c_{w}+2)}$, and define
\begin{equation}\label{eq gammma2}
\Gamma:= \Bigg(c_{s}H + (2\sqrt{2}c_{w}+2)\\
\times\frac{(2\sqrt{2}c_{w}H+c_{s}(\E[\sup_n|\mathcal{H}_{n-1}|^2]+2 H))}{2 K}\Bigg).
\end{equation}
We have
\begin{equation}\label{eq gammma}
\gamma_n \leq \left(1-K h/2 \right)\gamma_{n-1} + \Gamma h^{\beta+1}+\E[\mathcal{R}_{n-1}]. 
\end{equation}
We complete the proof by Lemma \ref{lem:gronwall} below.  
\end{proof}

\begin{lemma} \label{lem:gronwall} Let $a_n, g_n, c \geq 0$, $n \in \mathbb{N}$ be given.
Moreover, assume that $1+\lambda>0$. Then, if $a_n \in \mathbb{R}$, $n \in \mathbb{N}$, satisfies 
\[
  a_{n+1} \leq a_n(1+ \lambda) + g_{n} + c,  \quad n \geq 0 \,,
\] 
then  
\[
  a_n \leq   a_0 e^{n\lambda} + c \frac{e^{n\lambda}-1}{\lambda} +
  \sum_{j=0}^{n-1} g_{j} e^{((n-1) - j)\lambda}, \qquad n \geq 1 \,.  
\]
\end{lemma}

\begin{remark}
	Note that if we can choose $\widetilde{\beta} > \beta$ in \eqref{eq:cond_sum2} (which, as we will see in Section \ref{sec:ex}, is the case e.g. for Euler and implicit Euler schemes) then from Theorem \ref{th:convergence} we get
	\begin{equation*}
	 \E[|\yc{c}{{T/h,x_{0}}} - \xf{f}{{T/h,y_0}} |^2]  \leq |x_0-y_0|^2e^{-K/2T} + (2\Gamma/K + C_{\mathcal{R}})h^{\beta} \,.
	\end{equation*}
\end{remark}

\subsubsection{Optimal choice of parameters}
Theorem \ref{th:convergence} is fundamental in terms of applying the MLMC as it guarantees that the estimate for the variance in  \eqref{ml_ass} 
holds. In particular, we have the following  Lemma.
\begin{lemma} Assume that all the assumptions from Theorem \ref{th:convergence} hold. \label{lem:optimal}
Let $g(\cdot)$ be a Lipschitz function. Define
\[
 h_\ell=2^{-\ell}, \quad  T_{\ell} \sim -\frac{2\beta}{K} \left(\log{h_{0}}+\ell\log{2} \right), \quad \forall \ell\geq 0.   
\]
Then resulting MLMC variance is given by.
\[
\text{Var}[\Delta_{\ell}] \leq 2^{-\beta \ell}, \quad \Delta_{\ell}=g\left(x_{\frac{T_{\ell-1}}{h_{\ell-1}}}^{(f,\ell)}\right)-g\left(x_{\frac{T_{\ell-1}}{h_{\ell-1}}}^{(c,\ell)}\right)
\]
\end{lemma}
\begin{proof}
Since $g$ is a Lipschitz function and $$\E\left|x_{\frac{T_{\ell}-T_{\ell-1}}{h_{\ell}}}^{h_{\ell}}-x_{0}\right|^{2}<\infty,$$ the proof is a direct consequence of Theorem \ref{th:convergence}. 
\end{proof}  

\begin{remark}
\label{rem:complexity}
Unlike in the standard MLMC complexity theorem \cite{Giles2015Acta} where the cost of simulating single path is of order $\mathcal{O}(h_\ell^{-1})$,  here we have  $\mathcal{O}(h_\ell^{-1}|\log{h_\ell}|)$. This is due to the fact
that terminal times are increasing with levels. For the case $h_\ell=2^{-\ell}$ this results in cost per path
$\mathcal{O}(2^{-\ell}\ell)$ and does not exactly fit the complexity theorem in \cite{Giles2015Acta}.
Clearly in the case when MLMC variance decays with $\beta >1$ we still recover the optimal complexity of order 
$\mathcal{O}(\varepsilon^{-2})$. However, in the case $\beta =1$  following the  proof by Giles \cite{Giles2015Acta} one can see that the  complexity becomes $\mathcal{O}(\varepsilon^{-2}|\log{\varepsilon}|^{3})$.
\end{remark}

\begin{remark}
In the proof above we have assumed that $K$ is independent of $h$, while we have also used crude bounds in order not to deal directly 
with all the individual  constants, since these would be dependent on the numerical schemes used. 
\end{remark}

\begin{example}
In the case of the Euler-Maruyama method as we see from the analysis\footnote{As we will see there $m' \leq m$ depending on the size of $\nabla U(0)$} in Section \ref{subsec:Euler} $K=2m'-L^{2}h_{\ell}, \beta=2$,  while $\mathcal{R}_{n}=0,\mathcal{H}_{n}=L$.  Here $L$ is the Lipschitz constant of the drift $\nabla U(x)$. 
\end{example}
 \section{Examples of suitable methods}
\label{sec:ex}
In this section we present two (out of many) numerical schemes that fulfil the conditions of Theorem \ref{th:convergence}. In particular, we need to verify that our scheme is  $L^2$-regular in time, it has bounded numerical moments as in $\textbf{H0}$ and finally that it satisfies the one-step error estimates
 \eqref{eq:weak}-\eqref{eq:strong}. Note that for both methods discussed in this section we verify condition \eqref{e:conditionZn} with $h^2$ instead of $h$. However, since in \eqref{e:conditionZn} we consider $h \in (0,1)$, both \eqref{e:Euler:condZn}  and \eqref{e:implicit:condZn} imply \eqref{e:conditionZn}.

\subsection{Euler-Maruyama method}
\label{subsec:Euler}
We start by considering the explicit Euler scheme
\begin{equation} \label{eq:explicit_Euler}
S_{h,\xi}^{f}(x)=x+h\nabla U(x) +\sqrt{2h}\xi, 
\end{equation}
while $S^{f}=S^{c}$, \ie,  we are using the same numerical method for the fine and coarse paths. 
In order to be able to recover the integrability and regularity conditions we will need to impose further assumptions on the potential\footnote{this restriction will be alleviated in Section \ref{subsec:Euler_implicit} by means of more advanced integrators} $U$. In particular, additionally to assumption  \textbf{HU0},  we assume that

\begin{itemize}
\item[ \textbf{HU1}]  There exists constant $L$ such that for any $x,y \in \R^d$ 	
\[
| \nabla U(x)-\nabla U(y)|   \leq  L| x-y| 
\]
\end{itemize}
As a consequence of this assumption we have
\begin{equation} \label{eq:HU1_con}
| \nabla U(x)|   \leq  L| x|  + | \nabla U(0) |
\end{equation}
We can now prove the $L^{2}$-regularity in time  of the scheme. 

\paragraph{$L^2$- regularity}
Since regularity is a property of the numerical scheme itself and it does not relate with the coupling between fine and coarse levels, for simplicity 
of notation we prove things directly for 
\begin{equation}\label{defEuler}
x_{{n+1,x_{n}}}= S^{f}_{h,\xi_{n+1}}(x_{n}) .
\end{equation}
In particular, the following Lemma holds.
\begin{lemma}[$L^2$-regularity]\label{lemmaEulerL2regularity}
Let \textbf{HU0} and \textbf{HU1} hold. Then the explicit Euler scheme is $L^2$-regular, i.e.,
{\small
\begin{align}
	\E_{n-1}[|x_{n,x_{n-1}} - x_{n,y_{n-1}}|^2]\leq & (1 - (2m-L^2h)h ) |x_{n-1}-y_{n-1} |^2   \\
	\E_{n-1}[|Z_n|^2]   \leq &  	h^2 L^2 |x_{n-1} - y_{n-1} |^2 \label{e:Euler:condZn}
\end{align}}
\end{lemma}
\begin{proof}
The difference between the  Euler scheme taking values $x_{n-1}$ and $y_{n-1}$ at time $n-1$ is given by 
\begin{align*}
x_{n,x_{n-1}} - x_{n,y_{n-1}} = x_{n-1} - y_{n-1} 
 +  h( \nabla U(x_{n-1})  - \nabla U(y_{n-1})).
\end{align*}
This, along with \textbf{HU0} and \textbf{HU1} leads to
\begin{align*}
\E_{n-1}[(x_{n,x_{n-1}} - x_{n,y_{n-1}})^2] &= 
|x_{n-1} - y_{n-1}|^2
 + 2h \left\langle \nabla U(x_{n-1}) - \nabla U(y_{n-1}),x_{n-1}-y_{n-1}\right\rangle \\
 &+ | \nabla U(x_{n-1})  - \nabla U(y_{n-1}) |^{2} h^2  \\
& \leq  |y_{n-1}-x_{n-1}|^2 (1 - 2mh + L^2 h^2) \\
& =  |y_{n-1}-x_{n-1}|^2 (1 - (2m-L^2h)h ). 
\end{align*}
This proves  the first part of the lemma. Next, due to \textbf{HU1}
\begin{align*}
\E_{n-1}[|Z_n|^2]  = & h^2 \E_{n-1}[| \nabla U(x_{n-1})  - \nabla U(y_{n-1})  |^2]  
            \leq   	h^2 L^2 |x_{n-1} - y_{n-1} |^2. 
\end{align*}
\end{proof}

\paragraph{Integrability}
In the Lipschitz case we only require mean-square integrability.  This will become apparent when we analyse the  one-step error and \eqref{eq:weak} and \eqref{eq:strong} will hold with $p=1$.

\begin{lemma}[Integrability]\label{lem:int}\label{lemmaEulerIntegrability}
Let \textbf{HU0} and \textbf{HU1} hold. Then, 
\begin{align*}
	\E[|x_{n}|^2] &\leq \E|x_{0}|^2 \exp\{-(2 m'  - L^2h) nh \} 
	 + 2 (b | \nabla U(0) |^2+h) \frac{1-\exp\{-(2 m'  - L^2h) nh \}}{(2 m'  - L^2h) h }
\end{align*}
\end{lemma}
\begin{proof}
We have 
\begin{align*}
|x_{n}|^2 =  |x_{n-1}|^2 +  |\nabla U(x_{n-1})|^2 h^2 +2h \xi^{T}\xi  
+ 2 h  x^T_{n-1} \nabla U(x_{n-1}) +\sqrt{2h}x^T_{n-1} \xi 
+ \sqrt{2}h^{3/2} \xi^T \nabla U(x_{n-1}) 
\end{align*}
and hence using \eqref{eq:ing1}
\[
\IE |x_{n}|^2 \leq  \IE |x_{n-1}|^2 (1 - 2 m' h + L^2 h^2)  + 2b | \nabla U(0) |^2+2dh.
\]
We can now use Lemma \ref{lem:gronwall}
\begin{align*}
\E|x_{n}|^2 & \le  \E|x_{0}|^2 \exp\{-(2 m'  - L^2h) nh \}  + 2 (b | \nabla U(0) |^2+dh) \frac{1-\exp\{-(2 m'  - L^2h) nh \}}{(2 m'  - L^2h) h } 
\end{align*}
The proof for $q>2$ can be done in similar way by using the  binomial theorem. 
\end{proof}

\paragraph{One-step errors estimates}

Having proved $L^{2}$-regularity and integrability for the Euler scheme, we are now left with the task of proving inequalities \eqref{eq:weak} and 
\eqref{eq:strong} for Euler schemes coupled as in Algorithm \ref{alg:CouplingLangevinDiscretisation}. It is enough to prove the results for $n=1$.  We note that both $\xf{f}{0}=\yc{c}{0}=x$ and we have the following Lemma.  
\begin{lemma}[One-step errors] \label{lem:onestep}
Let \textbf{HU0} and \textbf{HU1} hold. 
Then the weak one-step distance between Euler schemes with time steps $\frac{h}{2}$ and $h$, respectively, is given by
\begin{equation} \label{eq:weak_euler}
| \E[ \xf{f}{{1,x}}- \yc{c}{{1,x}} ]	| 
\leq \frac{h^{3/2}}{2} L \left( \E \left[ \frac{\sqrt{h}}{2} \left(L|x| + | \nabla U(0) |\right)\right] + \sqrt{\frac{2d}{\pi}} \right).
\end{equation}
The one-step $L^2$ distance  can be estimated as
\begin{equation} \label{eq:strong_euler}
\mathbb{E}|  \xf{f}{{1,x}}- \yc{c}{{1,x}}  | ^{2}
\leq h^3 \frac{L^2}{4}\left(\frac{h}{2} ( |x|^{2} + | \nabla U(0) |^2  ) + d \right)
\end{equation}
If in addition to \textbf{HU0} and \textbf{HU1}, $U\in C^3$ and\footnote{Thanks to the  integrability conditions we could easily extend the analysis to the case where the derivatives are bounded by a polynomial of x.} 
\[
|  \partial^{2}U(x) | + | \partial^{3} U(x) | \leq C, \quad \forall x\in \R^d,
\]
then the order in $h$ of the weak error bound can be improved, i.e.,
\begin{equation} \label{eq:weak_euler1}
| \E[  \xf{f}{{1,x}}- \yc{c}{{1,x}}  ]	| \leq  C h^2  \E \big[ 
 |x | + h|x |^2 
 + |\nabla U(0)| + h |\nabla U(0)|^2  + d  \big].
\end{equation}
\end{lemma}

\begin{proof}
We calculate
\begin{align} \label{eq:euler_basis}
    \xf{f}{{1,x}}- \yc{c}{{1,x}}  \nonumber 
 & =   x+\frac{h}{2}\nabla U(x)+\sqrt{h}\xi_{1}+\frac{h}{2}\nabla U\left(x+\frac{h}{2}\nabla U(x)+\sqrt{h}\xi_{1}\right) \nonumber \\
 &+\sqrt{h}\xi_{2}
 	-\left(x+h\nabla U(x)+\sqrt{h}\left(\xi_{1}+\xi_{2}\right)\right) \nonumber\\
 & =   \frac{h}{2}\nabla U \left(x+\frac{h}{2}\nabla U(x)+\sqrt{h}\xi_{1}\right)-\frac{h}{2}\nabla U(x).
\end{align}
It then follows  from  \textbf{HU1} that 
\begin{equation*}
| \E[  \xf{f}{{1,x}}- \yc{c}{{1,x}}  ]	| \leq \frac{h^{3/2}}{2} L \E | \frac{\sqrt{h}}{2}\nabla U(x)+\xi_{1}|.
\end{equation*}
Furthermore, if we use \eqref{eq:HU1_con}, the triangle equality and the fact that $\IE | \xi_{1} |=\sqrt{\frac{2d}{\pi}}$, we obtain \eqref{eq:weak_euler}. If we now assume that $U\in C^3$, then for $\delta_t = x + t(\frac{h}{2}\nabla U(x)+\sqrt{h}\xi_{1})$,
$t \in [0,1]$, we write
\begin{align*}
 \nabla U \left(x+\frac{h}{2}\nabla U(x)+\sqrt{h}\xi_{1}\right) &=  \nabla U(x) 
 + \sum_{|\alpha|=1}\partial^{\alpha} \nabla U(x) \left(\frac{h}{2}\nabla U(x)+\sqrt{h}\xi_{1}\right)^{\alpha}  \\ 
&+ \sum_{|\alpha|=2}\int_0^1 (1-t) \partial^{\alpha} \nabla U(\delta_t)\, dt \, \left(\frac{h}{2}\nabla U(x)+\sqrt{h}\xi_{1}\right)^{\alpha},
\end{align*}
where we used multi-index notation. Consequently 
\begin{align*}
 \E \left[  \nabla U \left(x+\frac{h}{2}\nabla U(x)+\sqrt{h}\xi_{1} \right)-\nabla U(x) \right ] 
 \leq 
 C h^2  \E \left[ 
 (|x | + h|x |^2 + |\nabla U(0)| + h |\nabla U(0)|^2  + |\xi_{1}|^{2} |) \right],
\end{align*}
which, together with $\IE [|\xi_{1} |^2]=d$, gives \eqref{eq:weak_euler1}.
Equation \eqref{eq:strong_euler} trivially follows from \eqref{eq:euler_basis} by observing that 
\begin{eqnarray*}
    \mathbb{E}| \xf{f}{{1,x}}- \yc{c}{{1,x}} | ^{2}
 & \leq & L^{2}\frac{h^{2}}{4}\mathbb{E}| \frac{h}{2}\nabla U(x)+\sqrt{h}\xi_{1}| ^{2}\\
 & \leq & h^3 \frac{L^2}{4}\left(\frac{h}{2} ( |x|^{2} + | \nabla U(0) |^2  ) + d \right)
\end{eqnarray*}
\end{proof}
\begin{remark}
In the case of log-concave target the bias of MLMC using the Euler method can be explicitly quantified using the results from \cite{durmus2016ula}. 
\end{remark}

\subsection{Non-Lipschitz setting}
\label{subsec:Euler_implicit}
In the previous subsection we found out that in order to analyse the regularity and the one-step error of the explicit Euler 
approximation, we had to impose an additional assumption about $\nabla U(x)$ being globally Lipschitz.  This is necessary 
since in the absence of this condition Euler method is shown to be transient or even divergent
 \cite{RT96,Hutzenthaler2014}.  However, in many applications of interest this is a rather restricting condition. An example of this, is the  potential \footnote{One also may consider the case of products of distribution functions, where after taking the $\log$  one ends up with a polynomial in the different variables.} 
 \[
U(x) = -\frac{x^4}{4} -\frac{x^2}{2}.
\]
A standard way to deal with this is to use either an implicit scheme or specially designed explicit schemes  \cite{hutzenthaler2012numerical,szpruch2013v}. Here we will study only the case of implicit Euler.

\subsubsection{Implicit Euler method}

Here we will focus on the implicit Euler scheme
\[
x_{n}=x_{n-1}+h\nabla U(x_{n})+\sqrt{2h}\xi_{n}
\] 
We will assume that Assumption \textbf{HU0} holds and moreover replace \textbf{HU1} with
\begin{itemize}
\item[ \textbf{HU1'}] Let $k\geq 1$. For any $x,y \in \R^d$ there exists constant $L$ s.t	
\[
| \nabla U(x)-\nabla U(y)|   \leq  L(1+ |x|^{k-1} + |y|^{k-1} )| x-y| 
\]
\end{itemize} 
As a consequence of this assumption we have
\begin{equation} \label{eq:HU1_cona}
| \nabla U(x)|   \leq  L| x|^k  + | \nabla U(0) |
\end{equation}

\paragraph{Integrability}

Uniform in time bounds on the $p$-th moments of $x_n$ for all $p \geq 1$ can be easily deduced from the results in 
\cite{MR3011916,MR2972586}. Nevertheless,  for the convenience of the reader we will present the analysis of the regularity of the 
 scheme, where the effect of the implicitness of the scheme on the regularity should become quickly apparent.
 \paragraph{$L^2$- regularity}

\begin{lemma}[$L^2$-regularity]
Let \textbf{HU0} and \textbf{HU1'} hold. Then an implicit Euler scheme is $L^2$-regular, i.e.,
\begin{align}
	\E_{n-1}[(x_{n,x_{n-1}} - x_{n,y_{n-1}})^2] &\leq  (1 - 2mh ) (y_{n-1} - x_{n-1} )^2 + \mathcal{R}_{n-1},
\end{align}
and
\[
\sum_{k=0}^{\infty}\mathcal{R}_k \leq 0.
\]
Moreover,
\begin{equation}\label{e:implicit:condZn}
\E_{n-1} [ |Z_n|^2 ] \leq h^2 \mathcal{H}_{n-1} | x_{n-1}-y_{n-1} |^2 \,,
\end{equation}
where $\mathcal{H}_{n-1}$ is defined by \eqref{e:defH:implicit}.
\end{lemma}
\begin{proof}
The difference between the implicit Euler scheme taking values $x_{n-1}$ and $y_{n-1}$ time $n-1$ is given by 
\[
x_{n,x_{n-1}} - x_{n,y_{n-1}} = x_{n-1} - y_{n-1} +  h( \nabla U(x_{n})  - \nabla U(y_{n})).
\]
This, along with \textbf{HU0} and \textbf{HU1} leads to
\begin{align*}
 | x_{n,x_{n-1}} - x_{n,y_{n-1}} |^2 &= 
| x_{n-1} - y_{n-1}  |^2  
+ 2h \left\langle \nabla U(x_{n}) - \nabla U(y_{n}),x_{n}-y_{n}\right\rangle 
 - | \nabla U(x_{n})  - \nabla U(y_{n}) |^{2} h^2  \\
& \leq  | x_{n-1} - y_{n-1} |^2 
- 2mh | x_{n,x_{n-1}}-x_{n,y_{n-1}} |^2   \\
\end{align*}
This implies 
\begin{align*}
 | x_{n,x_{n-1}} - x_{n,y_{n-1}} |^2  
 \leq  | x_{n-1}-y_{n-1}|^2  \frac{1}{ 1+ 2mh}  
 \leq   | x_{n-1}-y_{n-1}|^2  \left(1 - \frac{2mh}{1+2mh}\right).
\end{align*}
Next we take
\begin{align*}
| x_{n,x_{n-1}} - y_{n,y_{n-1}} |^2
& \leq  | x_{n-1} - y_{n-1} |^2 
- 2mh | x_{n}-y_{n} |^2  \\
&=  ( 1 - 2mh )| x_{n-1} - y_{n-1} |^2 - 2mh( | x_{n}-y_{n} |^2 - | x_{n-1}-y_{n-1} |^2 ).
\end{align*}
 In view of Definition \ref{def:reg2} we define 
\[
\mathcal{R}_k := - 2mh( | x_{k}-y_{k} |^2 - | x_{k-1}-y_{k-1} |^2 ),
\]
and notice that 
\[
\sum_{k=1}^{n}\mathcal{R}_k = -2mh | x_{n}-y_{n} |^2 \leq 0.
\]
Hence the proof of the first statement in the Lemma is completed.  Now, due to \textbf{HU1'}
{\small
\begin{align*}
|Z_n|^2  =  h^2 | \nabla U(x_{n})  - \nabla U(y_{n})  |^2  
&\leq h^2 L^2(1+ |x_{n}|^{k-1} + |y_n|^{k-1} )^2| x_n-y_n |^2 \\
& \leq  h^2\left(1 - \frac{2mh}{1+2mh} \right) L^2(1+ |x_{n}|^{k-1} + |y_n|^{k-1} )^2| x_{n-1}-y_{n-1} |^2.
\end{align*}}
Observe that 
{\small
\begin{align*}
\E_{n-1}[| x_{n} |^2] &= |x_{n-1}|^2 
+ \E_{n-1}[
 2h \left\langle \nabla U(x_{n}) ,x_{n} \right\rangle 
 - | \nabla U(x_{n})   |^{2} h^2  ] + h\\
& \leq   | x_{n-1} |^2 
-  m h | x_{n}|^2 + h(| \nabla U (0)|^2 + 1). 
\end{align*}}
Consequently, 
\begin{align*}
\E_{n-1}[| x_{n} |^2] = &\frac{1}{1+mh}\left(  |x_{n-1} |^2 + h(| \nabla U (0)|^2 + 1)  \right).
\end{align*}
Similarly, it can be shown that $\E_{n-1}[| x_{n} |^k]$  can be expressed as a function of $|x_{n-1}|^k$ for $k > 2$, cf.\ \cite{MR3011916,MR2972586}. This in turn implies that there exists a constant $C > 0$ s.t.\
{\small\begin{equation}\label{e:defH:implicit}
\begin{split}
\mathcal{H}_{n-1} &= \E_{n-1} \left[ L^2 \left(1 - \frac{2mh}{1+2mh} \right)(1+ |x_{n}|^{k-1} + |y_n|^{k-1})^2 \right] \\
&\leq C (1 + |x_{n-1}|^{2(k-1)} + |y_{n-1}|^{2(k-1)}).
\end{split}
\end{equation}}
Due to uniform integrability of the implicit Euler scheme, \eqref{eq:cond_sum2}  holds.
\end{proof}

\paragraph{One-step errors estimates}
Having established integrability, estimating the one-step error follows exactly the same line of the argument as in 
 Lemma \ref{lem:onestep} and therefore we skip it.

\section{MLMC for SGLD}

\global\long\def\b{b}

\global\long\def\iid{\overset{\text{i.i.d.}}{\sim}}

\global\long\def\E{\mathbb{E}}

\label{sec:sgld}

\global\long\def\nobs{N}

\global\long\def\sobs{s}

\global\long\def\data{y}

In this section we discuss the multi-level Monte Carlo method for Euler schemes with inaccurate (randomised) drifts. Namely, we consider
\begin{equation}\label{eq:randomisedDriftEuler}
S_{h,\xi,\tau}(x) = x + hb(x,\tau) + \sqrt{2h}\xi \,,
\end{equation}
where $b: \mathbb{R}^d \times \mathbb{R}^k \to \mathbb{R}^d$ and an $\mathbb{R}^k$-valued random variable $\tau$ are such that 
\begin{equation}\label{eq:driftEstimator}
\mathbb{E}[b(x,\tau)] = \nabla U(x) \text{ for any } x \in \mathbb{R}^d \,.
\end{equation}
Our main application to Bayesian inference will be discussed in Subsection \ref{subsectionBayesian}. Let us now take a sequence $(\tau_n)_{n=1}^{\infty}$ of mutually independent random variables satisfying (\ref{eq:driftEstimator}). We assume that $(\tau_n)_{n=1}^{\infty}$ are also independent of the i.i.d. random variables $(\xi_n)_{n=1}^{\infty}$ with $\xi_n \sim \mathcal{N}(0,I)$. By analogy to the notation we used for the Euler scheme in (\ref{defEuler}), we will denote
\begin{equation}\label{eq:MLMCrandomisedDrift}
\bar{x}_{n+1, \bar{x}_n} = S^f_{h,\xi_{n+1},\tau_{n+1}}(\bar{x}_n) \,.
\end{equation}
In the sequel we will perform a one-step analysis of the scheme defined in \eqref{eq:MLMCrandomisedDrift} by considering the random variables 
\begin{equation}\label{eq:SGLDcoupling}
\begin{split}
\bar{x}^f_{1,x} &=S^f_{\frac{h}{2},\xi_{2},\tau^{f,2}}\circ S^f_{\frac{h}{2},\xi_{1},\tau^{f,1}}(x) \\
\bar{x}^{c}_{1,x} & =S^c_{h,\frac{1}{\sqrt{2}}\left(\xi_{1}+\xi_{2}\right),\tau^c}(x) \,,
\end{split}
\end{equation}
where $\xi_1$, $\xi_2 \sim \mathcal{N}(0,I)$ and $\tau^{f,1}$, $\tau^{f,2}$ and $\tau^{c}$ are $\mathbb{R}^k$-valued random variables satisfying \eqref{eq:driftEstimator}. In particular, $\tau^{f,1}$ and $\tau^{f,2}$ are assumed to be independent, but $\tau^c$ is not necessarily independent of $\tau^{f,1}$ and $\tau^{f,2}$.
We note that in (\ref{eq:SGLDcoupling}) we have coupled the noise between the fine and
the coarse paths synchronously, i.e., as in Algorithm \ref{alg:CouplingLangevinDiscretisation}. One question that naturally
occurs now is how one should choose to couple the random variables $\tau$ at different levels. In particular, in order for the condition with the telescopic sum to hold,
one needs to have
\begin{equation}
\mathcal{L}\left(\tau^{f,1}\right)=\mathcal{L}\left(\tau^{f,2}\right)=\mathcal{L}\left(\tau^{c}\right).\label{eq:SGLDTelescopic}
\end{equation}
We can of course just take $\tau^c$ independent of $\tau^{f,1}$ and $\tau^{f,2}$, but other choices are also possible, see Subsection \ref{subsectionBayesian} for the discussion in the context of the SGLD applied to Bayesian inference.

\begin{algorithm} [h]
	\begin{enumerate}
		\item Set $x_{0}^{(f,\ell)}=x_{0}$, then simulate according to $S_{h_{\ell},\xi,\tau}(x)$
		
		for $\frac{T_{\ell}-T_{\ell-1}}{h_{\ell}}$ steps with independent random input; 
		\item set $x_{0}^{(c,\ell)}=x_{0}$ and $x_{0}^{(f,\ell)}=x_{\frac{T_{\ell}-T_{\ell-1}}{h_{\ell}}}^{h_{\ell}}$,
		then simulate $(x_{\cdot}^{(f,\ell)},x_{\cdot}^{(c,\ell)})$ jointly
		according to 
		\begin{align*}
		\left(x_{k+1}^{(f,\ell)},x_{k+1}^{(c,\ell)}\right)=\Big(&S_{h_{\ell},\xi_{k,2},\tau^{f,2}_{k}}\circ S_{h_{\ell},\xi_{k,1},\tau^{f,1}_{k}}(x_{k}^{(f,\ell)}), S_{h_{\ell-1},\frac{1}{\sqrt{2}}\left(\xi_{k,1}+\xi_{k,2}\right),\tau^{c}_{k}}(x_{k}^{(c,\ell)})\Big).
		\end{align*}
		\item set $k_{\ell} := \frac{T_{\ell-1}}{h_{\ell-1}}$ and
		\[
		\Delta^{(i)}_{\ell}:=g\left( \left(x_{k_{\ell}}^{(f,\ell)} \right)^{(i)}\right)-g\left( \left( x_{k_{\ell}}^{(c,\ell)}\right)^{(i)} \right)
		\]
	\end{enumerate}
	\protect\caption{\label{alg:CouplingSGLDs}Coupling SGLD for $t_{i}\uparrow\infty$.}
\end{algorithm}

In order to bound the global error for our algorithm, we make the following assumptions on the function $b$ in (\ref{eq:randomisedDriftEuler}).
\begin{assumption}\label{assumptionSGLD}
	\begin{itemize}
		\item[(i)] There is a constant $\bar{L} > 0$ such that for any $\mathbb{R}^k$-valued random variable $\tau$ satisfying (\ref{eq:driftEstimator}) and for any $x$, $y \in \mathbb{R}^d$ we have
		\begin{equation}\label{eq:LipschitzEstimator}
		\mathbb{E} [|b(x,\tau) - b(y,\tau)|] \leq \bar{L} |x-y| \,. 
		\end{equation}
		\item[(ii)] There exist constants $\alpha_c$, $\sigma \geq 0$ such that for any $\mathbb{R}^k$-valued random variable $\tau$ satisfying (\ref{eq:driftEstimator}) and for any $h > 0$, $x \in \mathbb{R}^d$ we have
		\begin{equation}\label{eq:estimatorVariance}
		\mathbb{E} [|b(x,\tau) - \nabla U(x)|^2] \leq \sigma^2 (1 + |x|^2)h^{\alpha_c} \,.
		\end{equation}  
	\end{itemize}
\end{assumption}
Observe that conditions (\ref{eq:LipschitzEstimator}), (\ref{eq:estimatorVariance}) and Assumption \textbf{HU1} imply that for all random variables $\tau$ satisfying (\ref{eq:driftEstimator}) and for all $x \in \mathbb{R}^d$ we have
\begin{equation}\label{eq:LinearGrowthEstimator}
\mathbb{E}[|b(x,\tau)|^2] \leq \bar{L}_0(1 + |x|^2) 
\end{equation}
with
$\bar{L}_0 := \sigma^2 h^{\alpha_c} + 2 \max \left( L^2 , |\nabla U(0)|^2 \right)$,
cf.\ Section 2.4 in \cite{MajkaMijatovicSzpruch2018}.
For a discussion on how to verify condition (\ref{eq:estimatorVariance}) for a subsampling scheme, see Example 2.15 in \cite{MajkaMijatovicSzpruch2018}. By following the proofs of Lemmas \ref{lemmaEulerL2regularity} and \ref{lemmaEulerIntegrability}, we see that the $L^2$ regularity and integrability conditions proved therein hold for the randomised drift scheme given by (\ref{eq:randomisedDriftEuler}) as well, under Assumptions \textbf{HU0} and (\ref{eq:LipschitzEstimator}). Hence, in order to be able to apply Theorem \ref{th:convergence} to bound the global error for (\ref{eq:MLMCrandomisedDrift}), we only have to estimate the one step errors, i.e., we need to verify conditions (\ref{eq:weak}) and (\ref{eq:strong}) in an analogous way to Lemma \ref{lem:onestep} for Euler schemes.

\begin{lemma}\label{lemmaSGLDoneStep}
	Under Assumptions \ref{assumptionSGLD} and \textbf{HU1} there is a constant $C_1 = C_1(h,x)> 0$ given by 
	\begin{equation*}
	C_1 := \frac{1}{4} L \bar{L}_0^{1/2} h^{1/2} (1+|x|^2)^{1/2} + \frac{1}{2}L \sqrt{d}
	\end{equation*}
	 such that for all $h > 0$ we have
	\begin{equation}\label{randomisedEulerWeakError}
	\mathbb{E}[\bar{x}_{1,x}^f - \bar{x}_{1,x}^c] \leq C_1 h^{3/2} \,.
	\end{equation}
	Moreover, under the same assumptions there is a constant $C_2 = C_2(h,x)> 0$ given by
	\begin{equation*}
	C_2 := \frac{1}{4}\bar{L}^2 \bar{L}_0 h^{1 + (1-\alpha_c)^{+}}(1+|x|^2) + d\bar{L}^2 h^{(1-\alpha_c)^{+}} 
	+ 8 \sigma^2(1+|x|^2)h^{(\alpha_c - 1)^{+}}
	\end{equation*}
	such that for all $h > 0$ we have
	\begin{equation}\label{eq:randomisedEulerStrongError}
	\mathbb{E} [|\bar{x}_{1,x}^f - \bar{x}_{1,x}^c|^2] \leq C_2 h^{2 + \min(1, \alpha_c)} \,.
	\end{equation}
	\begin{proof}
		Note that we have
		\begin{align}\label{eq:randomisedOneStep1}
		\bar{x}_{1,x}^f - \bar{x}_{1,x}^c &= x + \frac{h}{2} b(x, \tau_1^f) + \sqrt{h} \xi_1 
		+ \frac{h}{2} b\left(x + \frac{h}{2} b(x, \tau_1^f) + \sqrt{h} \xi_1 , \tau_2^f\right)  \nonumber \\
		&+ \sqrt{h} \xi_2 - x - hb(x, \tau_1^c) -\sqrt{h} (\xi_1 + \xi_2)  \nonumber \\
		&= \frac{h}{2} b(x, \tau_1^f) + \frac{h}{2} b\left(x + \frac{h}{2} b(x, \tau_1^f) + \sqrt{h} \xi_1 , \tau_2^f\right)
		 - hb(x, \tau_1^c) \,.
		\end{align}
		By conditioning on all the sources of randomness except for $\tau_2^f$ and using its independence of $\tau_1^f$ and $\xi_1$, we show
		\begin{equation*}
		\mathbb{E}[\bar{x}_{1,x}^f - \bar{x}_{1,x}^c] \\
		= \frac{h}{2} \mathbb{E}\left[\nabla U\left( x + \frac{h}{2}b(x,\tau_1^f) + \sqrt{h} \xi_1 \right)\right] -\frac{h}{2} \nabla U(x) \,.
		\end{equation*}
		 Hence we have
		 \begin{equation*}
		 \mathbb{E}[\bar{x}_{1,x}^f - \bar{x}_{1,x}^c] \leq \frac{h}{2}L \mathbb{E} \left[\left| \frac{h}{2}b(x,\tau_1^f) + \sqrt{h} \xi_1 \right|\right]
		 \end{equation*}
		 and thus, using \eqref{eq:LinearGrowthEstimator} and Jensen's inequality, we obtain (\ref{randomisedEulerWeakError}). We now use (\ref{eq:randomisedOneStep1}) to compute
		\begin{align}\label{eq:randomisedOneStep2}
		\mathbb{E} [|\bar{x}_{1,x}^f - \bar{x}_{1,x}^c|^2] 
		&= h^2 \mathbb{E} \Big|\frac{1}{2}b(x,\tau^f_1) + \frac{1}{2}b(x,\tau^f_2) - \frac{1}{2}b(x,\tau^f_2) +
		+ \frac{1}{2}b\left(x + \frac{h}{2} b(x, \tau_1^f) + \sqrt{h} \xi_1 , \tau_2^f\right) - b(x,\tau_1^c)\Big|^2 \nonumber \\
		&\leq 2h^2 \mathbb{E}\left|\frac{1}{2}b(x,\tau^f_1) + \frac{1}{2}b(x,\tau^f_2) - b(x,\tau_1^c)\right|^2 \nonumber \\
		&+ \frac{1}{2} h^2 \mathbb{E} \left|b\left(x + \frac{h}{2} b(x, \tau_1^f) + \sqrt{h} \xi_1 , \tau_2^f\right) - b(x,\tau^f_2)\right|^2
		\end{align}
		Observe now that due to condition (\ref{eq:LipschitzEstimator}) the second term above can be bounded by
		\begin{align*}
		\frac{1}{2}h^2 \bar{L}^2 \mathbb{E}\left|\frac{h}{2} b(x, \tau_1^f) + \sqrt{h} \xi_1 \right|^2 
		& \leq \frac{1}{2}h^2 \bar{L}^2 \left( \frac{h^2}{2}\mathbb{E}|b(x,\tau_1^f)|^2 + 2h \mathbb{E}|\xi_1|^2 \right) \\
		& \leq \frac{1}{2}h^2 \bar{L}^2 \left( \frac{h^2}{2}\bar{L}_0(1 + |x|^2) + 2hd \right) \,,
		\end{align*}
		where in the last inequality we used (\ref{eq:LinearGrowthEstimator}). Moreover, the first term on the right hand side of (\ref{eq:randomisedOneStep2}) is equal to
		\begin{equation*}
		\begin{split}
		&2h^2 \mathbb{E}\Big|\frac{1}{2}b(x,\tau^f_1) - \frac{1}{2}\nabla U(x) + \frac{1}{2}\nabla U(x) - \frac{1}{2}b(x,\tau_1^c) \\    &+\frac{1}{2}b(x,\tau^f_2) - \frac{1}{2}\nabla U(x) + \frac{1}{2}\nabla U(x) - \frac{1}{2}b(x,\tau_1^c)\Big|^2 \\
		&\leq 2h^2 \Big( \mathbb{E}|b(x,\tau_1^f) - \nabla U(x)|^2 + 2\mathbb{E}|b(x,\tau_1^c) - \nabla U(x)|^2 + \mathbb{E}|b(x,\tau_2^f) - \nabla U(x)|^2 \Big) \\
		&\leq 8 \sigma^2 (1+ |x|^2) h^{2 + \alpha_c} \,,
		\end{split}
		\end{equation*}
		where in the last inequality we used (\ref{eq:estimatorVariance}). This finishes the proof of (\ref{eq:randomisedEulerStrongError}).
	\end{proof}
\end{lemma}

\begin{corollary}
	If $\alpha_c = 0$ in (\ref{eq:estimatorVariance}), then the Algorithm \ref{alg:CouplingSGLDs} based on the coupling given
	in Equation \eqref{eq:SGLDcoupling} with appropriately chosen $t_{i}$ has complexity
	$\varepsilon^{-2}|\log(\varepsilon)|^3.$ If $\alpha_c > 0$, then the algorithm has complexity $\varepsilon^{-2}$. \end{corollary}
\begin{proof}
	Because of Lemma \ref{lemmaSGLDoneStep} we can apply the 
	results of Section \ref{subsec:conv_anal}. In particular, if we choose $T_{\ell}$ according to Lemma \ref{lem:optimal}
	we thus for $\alpha_c = 0$ have $\beta=1$ in Theorem \ref{th:convergence} and then the complexity follows from Remark \ref{rem:complexity}. Similarly, for $\alpha_c > 0$ we have $\beta > 1$ and Remark \ref{rem:complexity} concludes the proof.\end{proof}

\subsection{Bayesian inference using MLMC for SGLD}\label{subsectionBayesian}

The main computational task in Bayesian statistics is the approximation
of expectations with respect to the posterior. The a priori uncertainty
in a parameter $x$ is modelled using a probability density $\pi_{0}(x)$
called the prior. Here we consider the case where for a
fixed parameter $x$ the data $\left\{ \data_{i}\right\} _{i=1,\dots,N}$
is supposed to be i.i.d. with density ${\pi(\data|x)}$. By Bayes'
rule the posterior is given by
\[
\pi(x)\propto\pi_{0}(x)\prod_{i=1}^{N}{\pi(\data_{i}|x)} \,.
\]
This distribution is invariant for the Langevin equation \eqref{eq:langevin} with 
\begin{equation}
\nabla U(x)=\nabla\log{\pi_{0}}(x)+\sum_{i=1}^{N}\nabla\log{\pi(\data_{i}|x)}.\label{eq:SGLDU}
\end{equation}
Provided that appropriate assumptions are satisfied for $U$ we can
thus use Algorithm \ref{alg:CouplingLangevinDiscretisation} with
Euler or implicit Euler schemes to approximate expectations with respect
to $\pi(dx)$. For large $N$ the sum in equation (\ref{eq:SGLDU}) becomes
a computational bottleneck. One way to deal with this is to replace the gradient by a lower cost stochastic approximation. In the following we apply our MLMC for SGLD framework to the recursion in Equation (\ref{eq:SGLD}) 
\[
x_{k+1} = x_{k}+h\left(\nabla\log{\pi_{0}(x_{k})}+\frac{N}{s}\sum_{i=1}^{s}\nabla\log{\pi(y_{\tau_{i}^{k}}\vert x_{k})}\right) + \sqrt{2h}\xi_{k} \,,
\]
where we take $\tau_{i}^{k}\iid\mathcal{U}\left(\{1,\dots,N\}\right)\text{ for }i=1,\dots,s$
where by $\mathcal{U}\left(\{1,\dots,N\}\right)$ we denote the uniform distribution on $1,\dots,N$ which corresponds to sampling
$s$ items with replacement from $1,\dots,N$. Notice that each step
only costs $s$ instead of $N$. We make the following assumptions on the densities ${\pi(\data|x)}$ and $\pi_0(x)$.
\begin{assumption} \label{assu:SGLDlipschitz}
	\begin{itemize}
		\item[(i)] 	Lipschitz conditions for prior and likelihood: There exist constants $L_0$, $L_1 > 0$ such that for all $i$, $x$, $y$
		\begin{eqnarray*}
			| \nabla\log\pi\left(\data_{i}\mid x\right)-\nabla\log\pi\left(\data_{i}\mid y\right)|  & \leq & L_1| x-y| \\
			| \nabla\log\pi_{0}\left(x\right)-\nabla\log\pi_{0}\left(y\right)|  & \leq & L_{0}| x-y| \,.
		\end{eqnarray*}
	\item[(ii)] Convexity conditions for prior and likelihood: There exist constants $m_{0} \geq 0$
	and $m_{y_i}\geq 0$ for $i = 1, \dots, N$ such that for all $i$, $x$, $y$
	\begin{align*}
	\log\pi_{0}(y) & \le  \log\pi_{0}(x)+\left\langle \nabla\log\pi_{0}\left(x\right),y-x\right\rangle  -\frac{m_{0}}{2}| x-y| ^{2}\\
	\log\pi\left(\data_{i}\mid y\right) & \le  \log\pi\left(\data_{i}\mid x\right)+\left\langle \nabla\log\pi\left(\data_{i}\mid x\right),y-x\right\rangle  -\frac{m_{\data_{i}}}{2}| x-y| ^{2}
	\end{align*}
	with $\inf_{i}(m_{0}+m_{\data_{i}})>0.$ 
	\end{itemize}
\end{assumption} 
We note that these conditions imply that the scheme given by (\ref{eq:SGLDcoupling}) with
\begin{equation*}
b(x,\tau) := \nabla\log{\pi_{0}(x)}+\frac{N}{s}\sum_{i=1}^{s}\nabla\log{\pi(y_{\tau_{i}}\vert x)}
\end{equation*}
for $x \in \mathbb{R}^d$, $\tau \in \mathbb{R}^s$, satisfies Assumptions \textbf{HU0}, \textbf{HU1} and (\ref{eq:LipschitzEstimator}). The value of the variance $\sigma$ of the estimator of the drift in (\ref{eq:estimatorVariance}) depends on the number of samples $s$, cf. Example 2.15 in \cite{MajkaMijatovicSzpruch2018}.

Regarding the coupling of $\tau^{f,1}$, $\tau^{f,2}$ and $\tau^c$, we have several possible choices.
We first take $s$ independent samples $\tau^{f,1}$ on the first fine-step
and another $s$ independent samples $\tau^{f,2}$ on the second fine-step. The following three choices of $\tau^c$
ensure that equation (\ref{eq:SGLDTelescopic}) holds.
\begin{itemize}
	\item[(i)] an independent sample of $\left\{ 1,\dots,N\right\} $ without replacement
	denoted as $\tau^c_{\text{ind}}$ called independent coupling;
	\item[(ii)] a draw of $s$ samples without replacement from $(\tau^{f,1},$ 
	$\tau^{f,2})$
	denoted as $\tau^c_{\text{union}}$ called union coupling; 
	\item[(iii)] the concatenation of a draw of $\frac{s}{2}$ samples without replacement
	from $\tau^{f,1}$ and a draw of $\frac{s}{2}$ samples without replacement
	from $\tau^{f,2}$ (provided that $s$ is even) denoted as $\tau^c_{\text{strat}}$
	called stratified coupling. 
\end{itemize}
We stress that any of these couplings can be used in Algorithm \ref{alg:CouplingSGLDs}. The problem of coupling the random variables $\tau$ between different levels in an optimal way will be further investigated in our future work.


\section{Numerical Investigations}
\label{sec:num}
In this section we perform numerical simulations that illustrate our theoretical findings. We start by studying an Ornstein-Uhlenbeck process in Section \ref{subsec:OU} using the explicit Euler method, while in Section \ref{subsec:log}  we study a Bayesian logistic
regression model using the SGLD.
\subsection{Ornstein Uhlenbeck  process}
\label{subsec:OU}
We consider the SDE 
\begin{equation} \label{eq:OU}
dX_{t}=-\kappa X_{t}dt+\sqrt{2}dW_{t},
\end{equation}
and its discretisation using the Euler method
\begin{equation} \label{eq:OU_num}
x_{n+1}=S_{h,\xi}(x_{n}), \quad S_{h,\xi}(x)=x-h\kappa x+\sqrt{2h}\xi.
\end{equation}
Equation \eqref{eq:OU} is ergodic with  its invariant measure being $\mathcal{N}(0,\kappa^{-1})$. Furthermore, it is possible to show that the Euler method 
\eqref{eq:OU_num} is similarly ergodic with its invariant measure \cite{Zyg11} being $\mathcal{N}\left(0,\frac{2}{2\kappa-\kappa^{2}h}\right)$. In Figure 
\ref{fig:OU1}, we plot the outputs of our numerical simulations using Algorithm \ref{alg:CouplingLangevinDiscretisation}. The parameter of interest 
here is the variance of the invariant measure $\kappa^{-1}$ which we try to approximate for different mean square error tolerances  $\varepsilon$.

 More precisely,  in Figure \ref{fig:OU1}a we see  the allocation of samples for various levels with respect to  $\varepsilon$, while in Figure 
\ref{fig:OU1}b we compare the computational cost of the algorithm as a function of the parameter $\varepsilon$. As we can see   the computational 
complexity grows as $\mathcal{O}(\varepsilon^{-2})$ as predicted by our theory (Here $\alpha=\beta=2$ in \eqref{eq:weak} and \eqref{eq:strong}).

Finally, in Figure \ref{fig:OU1}c we plot the approximation  of the variance $\kappa^{-1}$ from our algorithm. Note that this coincides with the choice 
$g(x)=x^{2}$ since the mean of the invariant measure is 0. As we can see as $\varepsilon$ becomes smaller, even though the estimator is in principle biased we get perfect agreement with the true value of the variance. 

\begin{figure}[htb]
\centering
\subfloat[
Numbers of samples on different levels for given accuracy ]{
\includegraphics[scale=.24]{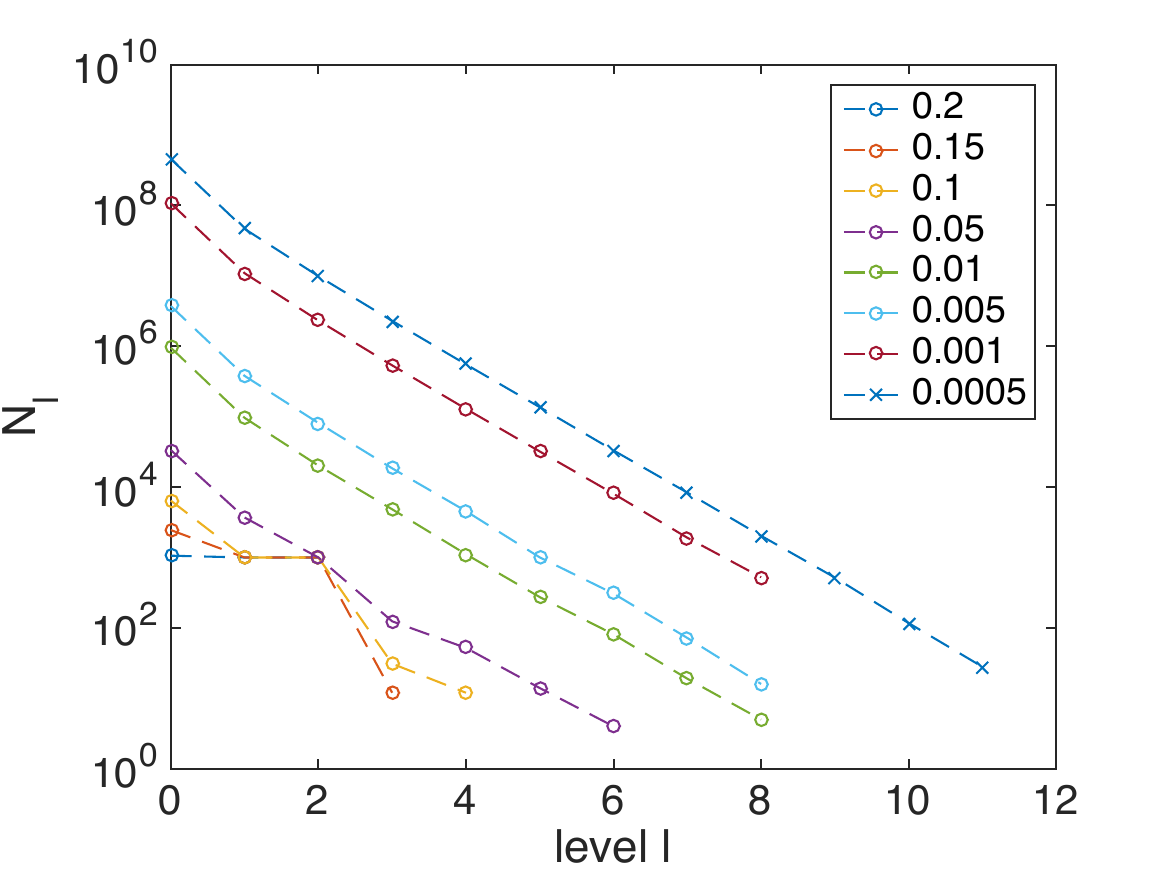}
} 
\subfloat[Cost vs accuracy]{
\includegraphics[scale=.24]{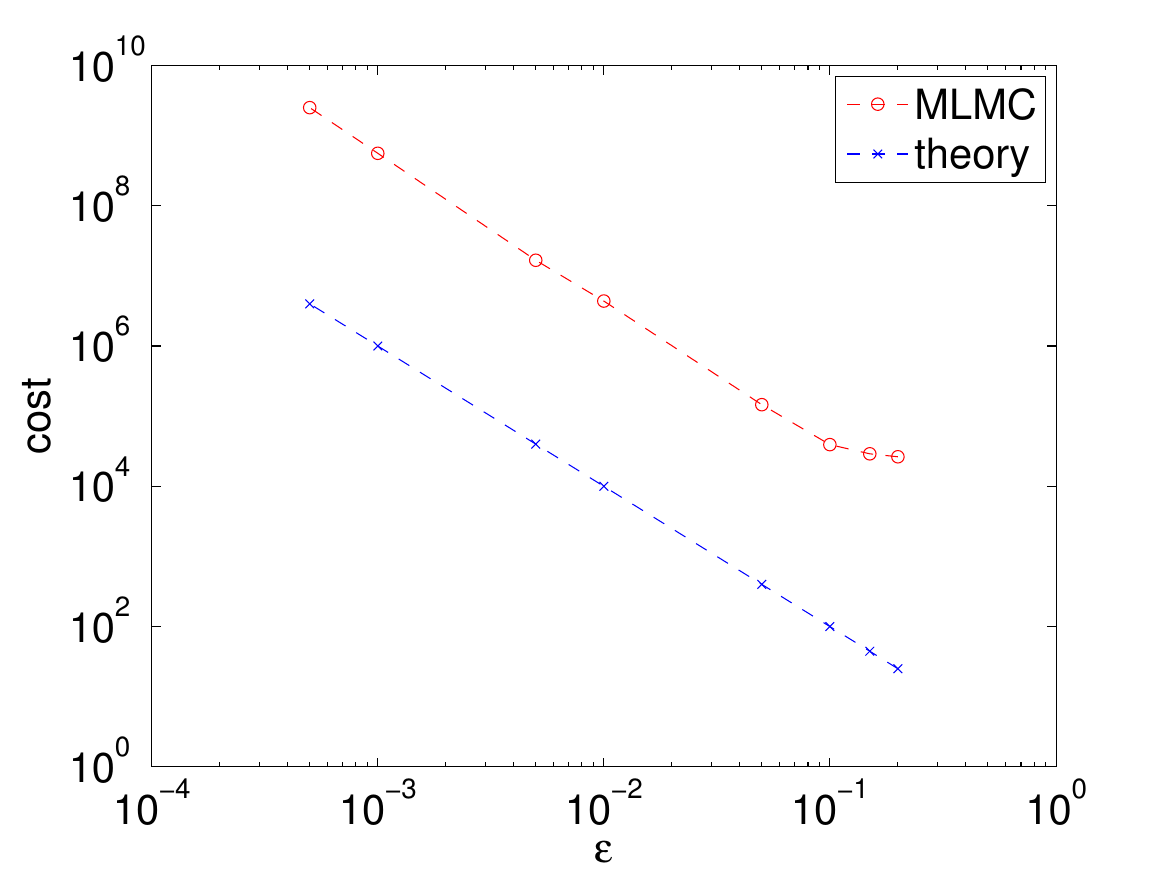}
}
\subfloat[Estimation of $\kappa^{-1}$]{
\includegraphics[scale=.24]{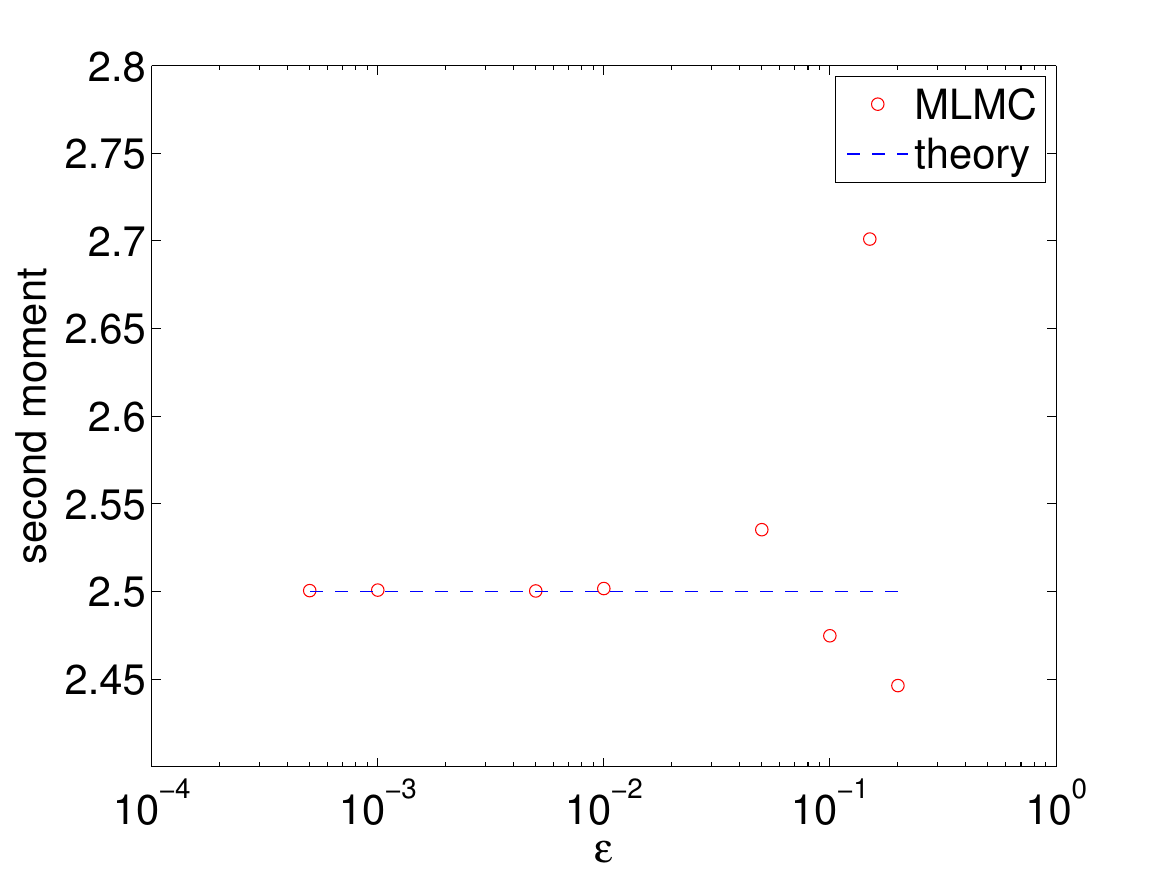}
}
\caption{MLMC results for \eqref{eq:OU} for $g(x)=x^{2}$ and $\kappa=0.4$.}
\label{fig:OU1}
\end{figure}

\subsection{Non-Lipschitz }
\label{subsec:non_lip}
We consider the SDE 
\begin{equation} \label{eq:non_lip}
dX_{t}=-\left(X^{3}_{t} + X_{t} \right)dt+\sqrt{2}dW_{t},
\end{equation}
and its discretisation using the implicit Euler method
\begin{equation} \label{eq:OU_num_imp}
x_{n+1}=x_{n}-h\left(x^{3}_{n+1}+x_{n+1} \right)+\sqrt{2h}\xi_{n}.
\end{equation}
In Figure 
\ref{fig:non_lip}, we plot the outputs of our numerical simulations using Algorithm \ref{alg:CouplingLangevinDiscretisation}. The parameter of interest 
here is the second moment of the invariant measure $\int_{\R}x^{2} \exp \left(-\frac{1}{4}x^{4}-\frac{1}{2}x^{2}\right) dx$ which we try to approximate for different mean square error tolerances  $\varepsilon$.

More precisely,  in Figure \ref{fig:non_lip}a we see  the  allocation of samples for various levels with respect to  $\varepsilon$, while in Figure 
\ref{fig:non_lip}b we compare the computational cost of the algorithm as a function of the parameter $\varepsilon$. As we can see   the computational 
complexity grows as $\mathcal{O}(\varepsilon^{-2})$ as predicted by our theory (Here $\alpha=\beta=2$ in \eqref{eq:weak} and \eqref{eq:strong}).

Finally, in Figure \ref{fig:non_lip}c we plot the approximation  of the second moment of the invariant measure from our algorithm. As we can see as $\varepsilon$ becomes smaller, even though the estimator is in principle biased we get perfect agreement with the true value\footnote{which has been calculated using high order quadrature} of the second moment.

\begin{figure}[htb]
\centering
\subfloat[
Numbers of samples on different levels for given accuracy ]{
\includegraphics[scale=.24]{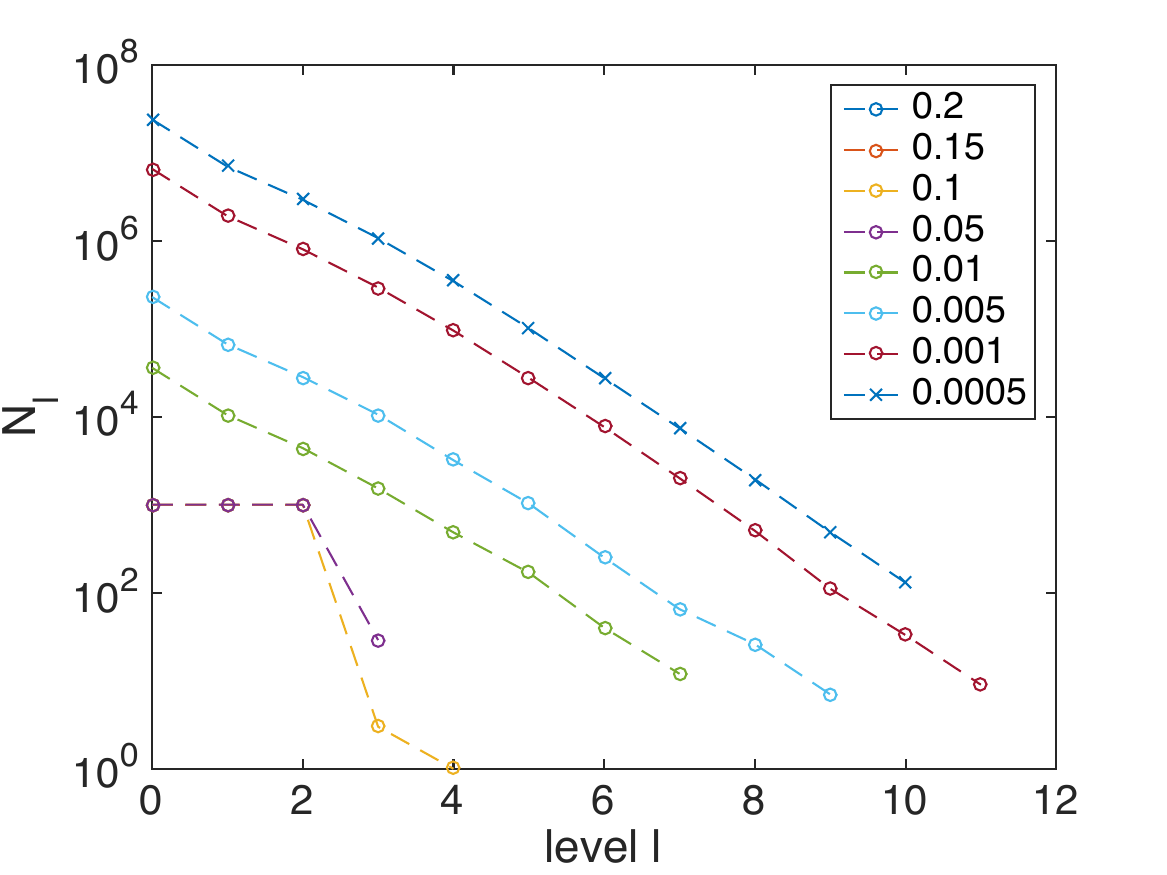}
} 
\subfloat[Cost vs accuracy]{
\includegraphics[scale=.24]{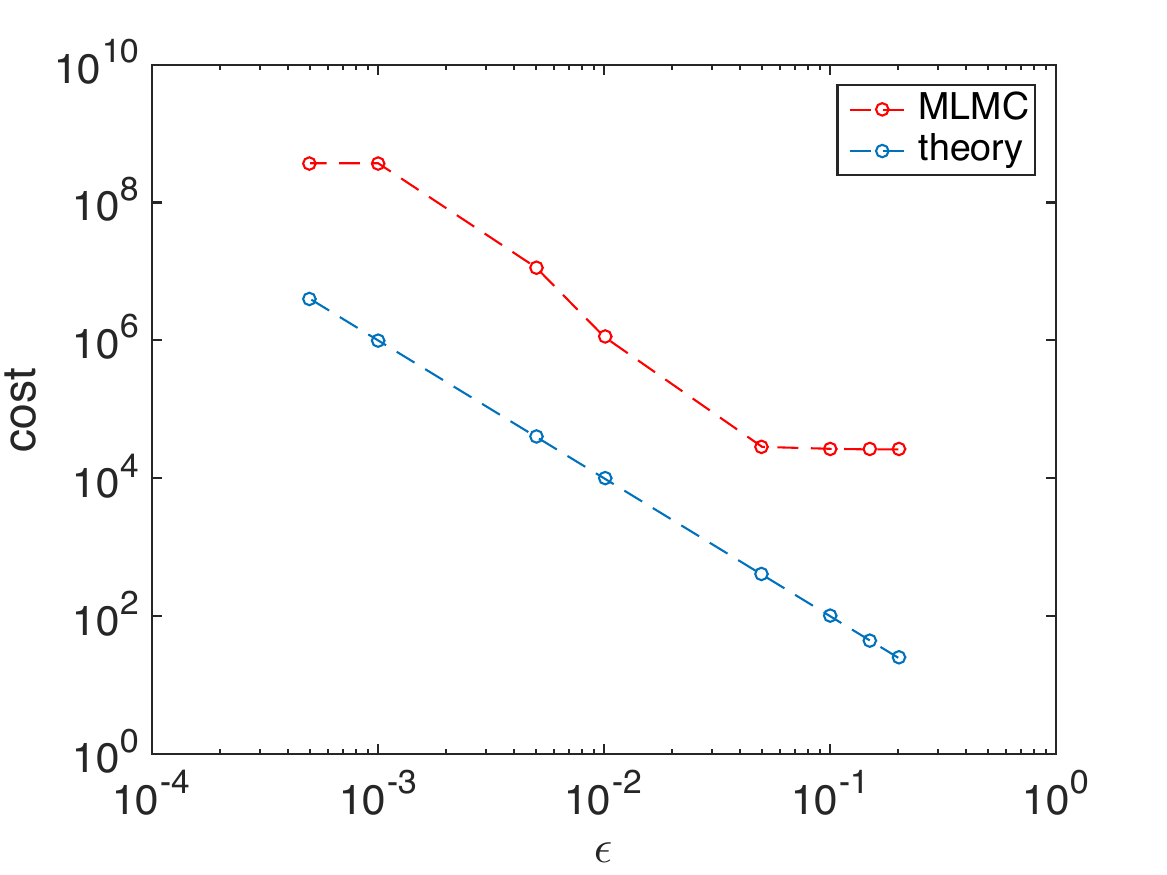}
}
\subfloat[Estimation of $\kappa^{-1}$]{
\includegraphics[scale=.24]{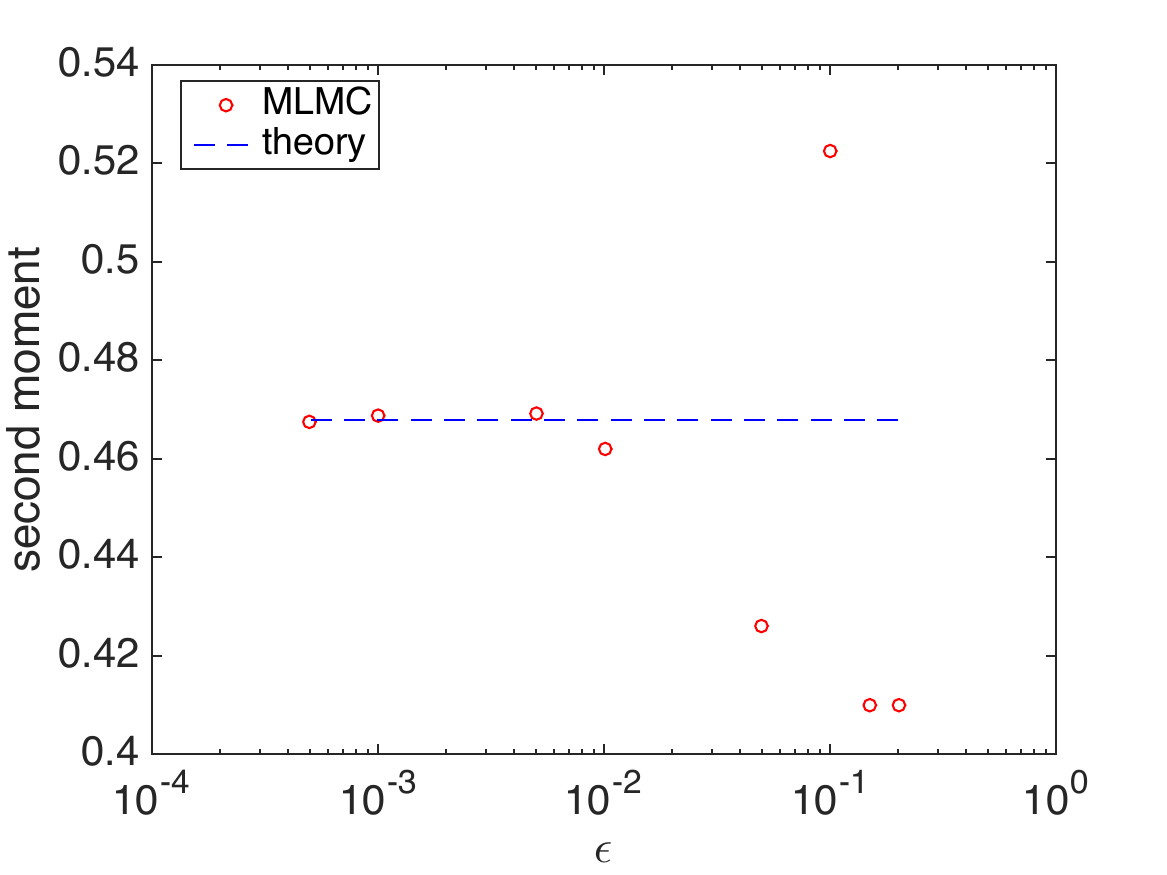}
}
\caption{MLMC results for \eqref{eq:non_lip} for $g(x)=x^{2}$.}
\label{fig:non_lip}
\end{figure}

\subsection{Bayesian logistic regression}
\label{subsec:log}
\global\long\def\logit{f}

\global\long\def\cova{\iota}

In the following we present numerical simulations for a binary Bayesian logistic
regression model. In this case the data $\data_{i}\in\{-1,1\}$ is modelled
by 
\begin{equation}
p(\data_{i}\vert\iota_{i},x)=\logit(y_{i}x^{t}\cova_{i})\label{eq.logistic}
\end{equation}
where $\logit(z)=\frac{1}{1+\exp(-z)}\in[0,1]$ and $\cova_{i}\in\mathbb{R}^{d}$
are fixed covariates. We put a Gaussian prior $\mathcal{N}(0,C_{0})$
on $x$, for simplicity we use $C_{0}=I$ subsequently. By Bayes'
rule the posterior density $\pi(x)$ satisfies 
\[
\pi(x)\propto\exp\left(-\frac{1}{2}| x|_{C_{0}}^{2}\right)\prod_{i=1}^{N}\logit(y_{i}x^{T}\cova_{i}).
\]
We consider $d=3$ and $N=100$ data points and choose the covariate
to be 
\[
\cova=\left(\begin{array}{ccc}
\cova_{1,1} & \cova_{1,2} & 1\\
\cova_{2,1} & \cova_{2,2} & 1\\
\vdots & \vdots & \vdots\\
\cova_{100,1} & \cova_{100,2} & 1
\end{array}\right)
\]
for a fixed sample of $\cova_{i,j}\overset{\text{i.i.d.}}{\sim}\mathcal{N}\left(0,1\right)$
for $i=1,\dots, 100$ and $j = 1, 2$.

In Algorithm \ref{alg:CouplingSGLDs} we can choose the starting position $x_0$. It is reasonable to start the path of the individual SGLD trajectories at the mode of the target distribution (heuristically this makes the distance $\E \norm{x_{0}^{(c,\ell)}-x_{0}^{(f,\ell)}}$  in step 2 in Algorithm \ref{alg:CouplingSGLDs}  small). That is we set the $x_0$  to be the maximum a posteriori estimator (MAP)
\[
x_{0}=\text{argmax}\:\exp\left(-\frac{1}{2}| x |_{C_{0}} ^{2}\right)\prod_{i=1}^{N}\logit(y_{i}x^{T}\cova_{i})
\]
which is approximated using the Newton-Raphson method.
Our numerical results are described in Figure \ref{fig:Logistic}. In particular,
in Figure \ref{fig:Logistic}a  we illustrate the behaviour of the coupling by plotting an estimate of the average distance during the joint evolution in 
step 2 of Algorithm \ref{alg:CouplingSGLDs}. The behaviour in this figure agrees qualitatively with the statement of Theorem \ref{th:convergence}, 
as $T$ grows there is an initial exponential decay up to an additive constant. For the simulation we used $h_0=0.02$, $T_\ell=3(\ell+1)$ and 
$s=20$.
Furthermore, in Figure \ref{fig:Logistic}b we plot $\text{CPU-time}\times \varepsilon^2 $ against $\varepsilon$ for  the estimation of the mean. The 
objective here is to estimate the mean square distance from the MAP estimator $x_0$ and the posterior that is $\int\norm{x-x_0}^2 \pi(x)dx$. Again, 
after some initial transient where $\text{CPU-time}\times \varepsilon^2$ decreases,  we see that we get a quantitive agreement with our theory 
since the $\text{CPU-time}\times \varepsilon^2$ increases in a 
logarithmic way  in the limit of $\varepsilon$ going to zero.

\begin{figure}[htb]
\centering
\subfloat[Coupling difference.]{
\includegraphics[scale=.23]{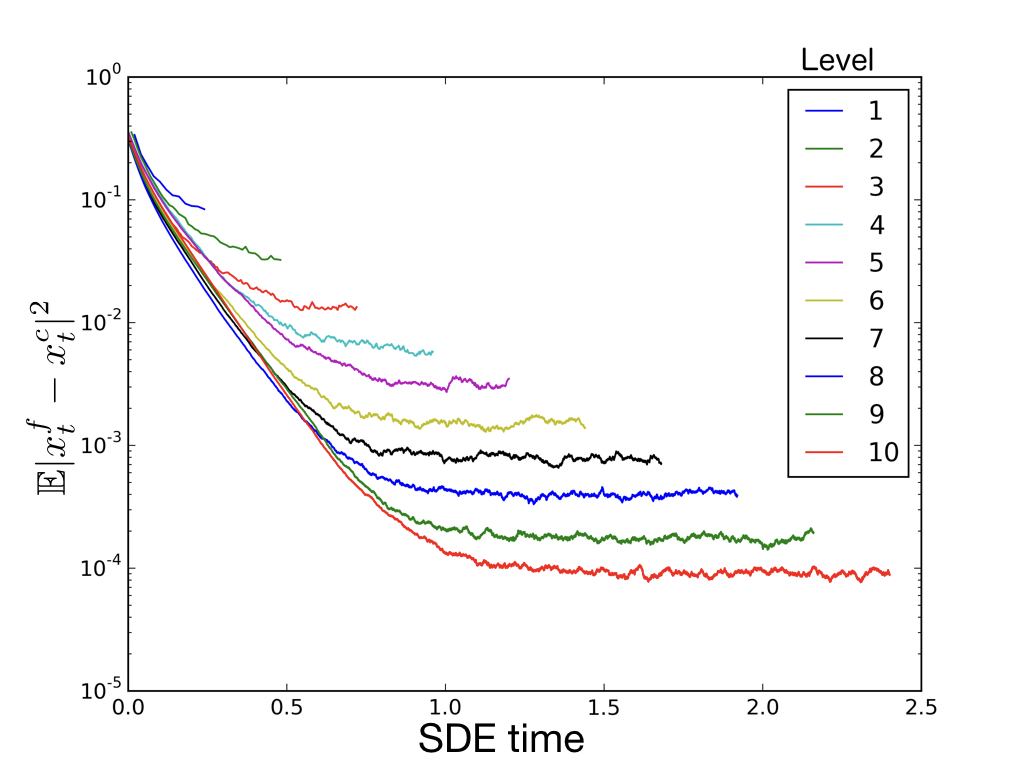}
} 
\subfloat[Cost of MLMC.]{
\includegraphics[scale=.25]{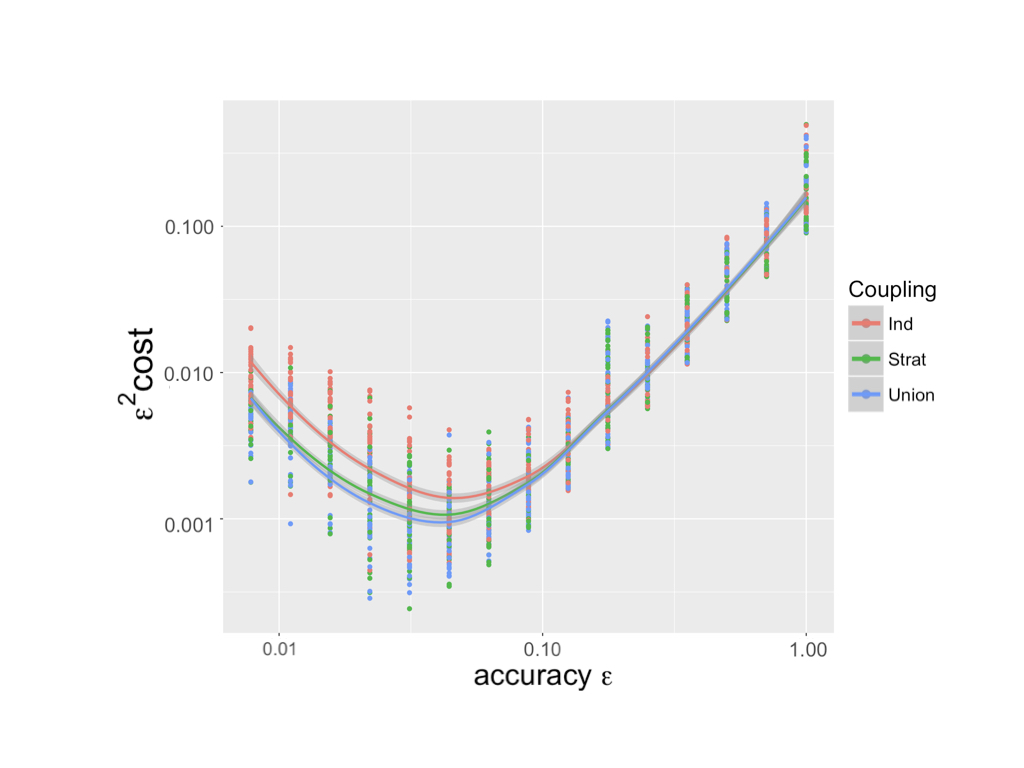}
}
\caption{(a) Illustration of the joint evolution in step
2 of Algorithm \ref{alg:CouplingSGLDs} for the union coupling, (b) Cost of MLMC (sequential CPU time) SGLD
for Bayesian Logistic Regression for decreasing accuracy parameter
$\varepsilon$ and different couplings}
\label{fig:Logistic}
\end{figure}

\section*{Acknowledgement}

MBM is supported by the EPSRC grant EP/P003818/1. KCZ is partially  supported  by the Alan Turing
Institute under the EPSRC grant EP/N510129/1. We would like to thank the referees for their careful reading of our manuscript and for numerous useful suggestions.

\bibliographystyle{spbasic}     
\def\cprime{$'$} \def\cprime{$'$} \def\cprime{$'$} \def\cprime{$'$}
  \def\cprime{$'$}


\appendix

\end{document}